\documentclass[11pt]{amsart}

\usepackage{amssymb,amsfonts}
\usepackage{hyperref}
\usepackage{mdwlist}

\usepackage{amssymb,textcomp}
\usepackage{enumerate}

\usepackage[utf8]{inputenc}
\usepackage[T1]{fontenc}

\usepackage[mathscr]{eucal}

\usepackage{hyperref}
 \hypersetup{
     colorlinks=true,
     linktocpage=true,
     linkcolor=red,
     filecolor=blue,
     citecolor = blue,
     urlcolor=cyan,
     }

\usepackage[a4paper, twoside=false, vmargin={2cm,3cm}, includehead]{geometry}

\usepackage{comment}
\newtheorem{lemma}{Lemma}[section]
\newtheorem{theorem}[lemma]{Theorem}
\newtheorem*{theorem*}{Theorem}
\newtheorem{corollary}[lemma]{Corollary}
\newtheorem{question}{Question}
\newtheorem{proposition}[lemma]{Proposition}
\newtheorem*{proposition*}{Proposition}

\newtheorem{conjecture}{Conjecture}

\newtheorem*{problem*}{Problem}

\theoremstyle{definition}
\newtheorem*{claim*}{Claim}

\newtheorem*{definition}{Definition}

\newtheorem*{example}{Example}

\newtheorem*{remark}{Remark}
\newtheorem*{remarks}{Remarks}

\newcommand{\C}{{\mathbb C}}
\newcommand{\E}{{\mathbb E}}

\newcommand{\N}{{\mathbb N}}

\newcommand{\R}{{\mathbb R}}

\newcommand{\T}{{\mathbb T}}
\newcommand{\Z}{{\mathbb Z}}

\newcommand{\CC}{{\mathcal C}}
\newcommand{\CD}{{\mathcal D}}

\newcommand{\CI}{{\mathcal I}}

\newcommand{\CP}{{\mathcal P}}

\newcommand{\CX}{{\mathcal X}}
\newcommand{\CY}{{\mathcal Y}}

\newcommand{\CZ}{{\mathcal Z}}

\newcommand{\uh}{{\underline{h}}}
\newcommand{\un}{{\underline{n}}}
\newcommand{\um}{{\underline{m}}}

\newcommand{\veps}{\varepsilon}
\newcommand{\eps}{\epsilon}
\newcommand{\ueps}{{\underline{\epsilon}}}

\newcommand{\norm}[1]{\left\Vert #1\right\Vert}
\newcommand{\nnorm}[1]{\lvert\!|\!| #1|\!|\!\rvert}

\DeclareMathOperator{\spec}{Spec}

\newcommand{\abs}[1]{\mathopen{}\left| #1\mathclose{}\right|}

\newcommand{\brac}[1]{\mathopen{}\left( #1 \mathclose{}\right)}

\begin{document}
	\title[Degree lowering for ergodic averages along arithmetic progressions]{Degree lowering for ergodic averages along arithmetic progressions}

	\thanks{The authors were supported  by the Hellenic Foundation for Research and Innovation, Project: 1684.  }
	
	\author{Nikos Frantzikinakis and Borys Kuca}
	\address[Nikos Frantzikinakis]{University of Crete, Department of mathematics and applied mathematics, Voutes University Campus, Heraklion 71003, Greece} \email{frantzikinakis@gmail.com}
	
	\address[Borys Kuca]{University of Crete, Department of mathematics and applied mathematics, Voutes University Campus, Heraklion 71003, Greece}
	\email{boryskuca@uoc.gr}
	\begin{abstract}
We examine the limiting behavior of multiple ergodic averages associated with arithmetic progressions whose differences are elements of a fixed integer sequence. For each $\ell$, we give necessary and sufficient conditions under which averages of length $\ell$ of the aforementioned form  have the same limit as averages of $\ell$-term arithmetic progressions. As a corollary, we derive a sufficient condition for the presence of arithmetic progressions with length $\ell+1$ and restricted differences in dense subsets of integers. These results are a consequence of the following general theorem: in order to verify that a multiple ergodic average is controlled by the degree $d$ Gowers-Host-Kra seminorm, it suffices to show that it is controlled by \textit{some} Gowers-Host-Kra seminorm, and that the degree $d$ control follows whenever we have degree $d+1$ control.
The proof relies on an elementary inverse theorem for the Gowers-Host-Kra seminorms involving  dual functions, combined with novel estimates on averages of seminorms of dual functions. We  use these estimates to obtain a higher order variant of the degree lowering argument previously used to cover averages that converge to the product of integrals.
	\end{abstract}

	
	\maketitle
	
	\setcounter{tocdepth}{1}
	\tableofcontents
\section{Introduction}

In the seminal work \cite{Fu77}, Furstenberg proved that for every $\ell\in\N$, system\footnote{For us, a \textit{system} is always an invertible Lebesgue measure-preserving dynamical system.}
$(X, \CX, \mu, T)$ and  set $A\in\CX$ of positive measure $\mu(A)>0$, we have the estimate
\begin{align}
	\liminf_{N\to\infty}\frac{1}{N}\sum_{n=1}^N \mu(A\cap T^{-n}A\cap\cdots \cap T^{-\ell n}A) >0;
\end{align}
in particular there exists $n\in\N$ for which $\mu(A\cap T^{-n}A\cap\cdots \cap T^{-\ell n}A)>0$. He subsequently used this multiple recurrence result to give an alternative proof of the Szemer\'edi theorem on the existence of arithmetic progressions in dense subsets of integers \cite{Sz75}.
We examine the following related problem. We call subsets of $\Z$ with positive upper density {\em dense}.
\begin{question}\label{Q: patterns}
	Given $\ell\in \N$,  for which  sequences $a\colon\N\to\Z$ do  dense subsets of $\Z$ contain the pattern
	\begin{align}\label{E: AP with a(n)}
		m, \; m+ a(n),\; \ldots,\; m+ \ell a(n)
	\end{align}
	for some $n\in\N$, i.e. an arithmetic progression of some fixed length $\ell+1$ and common difference $a(n)$, or more general patterns of the form
	\begin{align}\label{E: linear pattern}
		m, \; m+ k_1 a(n),\; \ldots,\; m+ k_\ell a(n)
	\end{align} for fixed $\ell\in\N$ and $k_1, \ldots, k_\ell\in\Z$?
\end{question}


The presence of patterns \eqref{E: AP with a(n)} and \eqref{E: linear pattern} in dense subsets of $\Z$  has been established for many natural families of sequences $a$: the best known example is the polynomial Szemer\'edi theorem of Bergelson and Leibman~\cite{BL96} that covers the case of polynomials
	with zero constant terms, but there are also other examples coming from generalized polynomial sequences,  Hardy field sequences of polynomial growth,  the primes shifted by 1 (or -1), and random sequences of integers, to name a few.  We refer the reader to \cite[Section~3.1]{Fr11} for references and further examples.
	 A recurrent theme in these works has been to show that under suitable conditions on the system, the averages
\begin{align}\label{E: AP average along a}
	\frac{1}{N}\sum_{n=1}^N T^{k_1 a(n)}f_1\cdots T^{k_\ell a(n)}f_\ell
\end{align}
converge in $L^2(\mu)$ as $N\to \infty$ and have the same limit as the averages
\begin{align}\label{E: AP average}
	\frac{1}{N}\sum_{n=1}^N T^{k_1 n}f_1\cdots T^{k_\ell n}f_\ell
\end{align}
for all functions $f_1, \ldots, f_\ell\in L^\infty(\mu)$ and $k_1,\ldots, k_\ell\in \Z$.
The point is that since  the limiting behavior of the averages \eqref{E: AP average} has been studied in great length (for example in  \cite{BHK05,Fu77, HK05,Zi05,Zi07}),  this   allows for an in depth understanding of  multiple recurrence and convergence properties of   the more complicated averages \eqref{E: AP average along a}.
This motivates the following question.
\begin{question}\label{Q: limits}
	Let $\ell\in\N$. For which  sequences $a\colon\N\to\Z$ and systems $(X, \CX, \mu, T)$ do the $L^2(\mu)$ limits of \eqref{E: AP average along a} and \eqref{E: AP average} as $N\to\infty$ coincide for all  $k_1, \ldots, k_\ell\in\Z$?
\end{question}

In this paper, we address both aforementioned questions: we give necessary and sufficient conditions for the limits of the  averages \eqref{E: AP average along a} and \eqref{E: AP average} to be equal, and from this we deduce a sufficient condition on the sequence $a$ that guarantees the presence of patterns \eqref{E: AP with a(n)} in dense subsets of integers.
These and other results are applications of the much more general Theorem \ref{T:seminormdrop} that gives necessary and sufficient  conditions for the $L^2(\mu)$ limiting behavior of
\begin{align}\label{E:averageseveral}
	\frac{1}{N}\sum_{n=1}^N T^{a_1(n)}f_1\cdots T^{a_\ell(n)}f_\ell
\end{align}
over a fixed system $(X, \CX, \mu, T)$ to be controlled by a Gowers-Host-Kra seminorm of a fixed degree $d$. We prove that Gowers-Host-Kra seminorms of degree $d$ control the averages \eqref{E:averageseveral} under two conditions:
\begin{enumerate}
	\item  the averages \eqref{E:averageseveral} are controlled by \textit{some} Gowers-Host-Kra seminorm, potentially of a very high degree;
	\item  they are controlled by a degree $d$ seminorm whenever they are controlled by a degree $d+1$ seminorm (i.e. when the functions are $\CZ_d$-measurable).
\end{enumerate}
When $d=0$, this {\em degree reduction property} implies that to prove the $L^2(\mu)$ convergence of \eqref{E:averageseveral} to the product of integrals of $f_1, \ldots, f_\ell$ for ergodic $T$,
it suffices to establish control of \eqref{E:averageseveral} by some Gowers-Host-Kra seminorm and show that the $L^2(\mu)$ limit of \eqref{E:averageseveral} vanishes whenever $f_1, \ldots, f_\ell$ are eigenfunctions of $T$ and at least one of them has an eigenvalue different than $1$.  This has been proved by the first author \cite{Fr21}, and the extension to the setting of commuting transformations has recently been established by both authors \cite{FrKu22a}. Theorem \ref{T:seminormdrop} can thus be viewed as a ``higher-step'' generalization of the main result from \cite{Fr21}. Its proof relies on an extension of the degree lowering argument from \cite{Fr21}, originally introduced by Peluse~\cite{P19, P20} and Peluse and Prendiville~\cite{PP19} in finitary works on the polynomial Szemer\'edi theorem. Importantly, even though we deal with Gowers-Host-Kra seminorms of arbitrary degree, our argument is elementary in that it does not make use of the deep Host-Kra structure theorem from \cite{HK05}. Rather, it relies on a simple identity relating a Gowers-Host-Kra seminorm of a function to an appropriate dual function and a novel estimate on averages of seminorms of dual functions (Proposition \ref{P:keyestimate}).  On the other hand, it is the combination of the new degree reduction property of Theorem~\ref{T:seminormdrop} and the Host-Kra structure theorem that enables us to establish most of the applications given in Sections~\ref{SSS:APs2}-\ref{SS:1.5}.

\bigskip
\noindent {\bf Acknowledgments.} We would like to thank
John Griesmer and Or Shalom   for helpful comments. We would also like to thank the referee for pointing out  a mistake in Section~\ref{SS:identities}.

\smallskip

\section{Definitions of various good properties}\label{S:good properties}
In this section, we gather several notions used in the statements of our main results, which are given in the next section.
\subsection{Convergence along APs}
We start with notions related to mean  convergence properties of averages of the form \eqref{E:averagesGeneral}, which we call averages along $\ell$-term arithmetic progressions.
\begin{definition}[Mean convergence along APs]
	We say that a sequence $a\colon \N\to \Z$ is
	{\em good for mean convergence along  $\ell$-term arithmetic progressions (APs)} for the system $(X, \CX, \mu,T)$ if  for all $k_1,\ldots, k_\ell\in \Z$
	and $f_1,\ldots, f_\ell\in L^\infty(\mu)$, the identity
	\begin{equation}\label{E:identityAPs}
		\lim_{N\to\infty}	\frac{1}{N}\sum_{n=1}^N  \, T^{k_1a(n)}f_1  \cdots T^{k_\ell a(n)}f_\ell=	\lim_{N\to\infty}	\frac{1}{N}\sum_{n=1}^N  \, T^{k_1n}f_1 \cdots T^{k_\ell n}f_\ell
	\end{equation}
	holds where both limits are taken in $L^2(\mu)$ (the second limit is known to exist by \cite{HK05}).
\end{definition}
\begin{remark}
	Using an ergodic decomposition argument, we see that if a sequence is good for mean convergence along $\ell$-term APs for every ergodic system, then the same property holds for every system.
\end{remark}
Note that if a sequence is good for mean convergence along $\ell$-term APs for a specific system, then Furstenberg's multiple recurrence result~\cite{Fu77}  implies that it is good for recurrence along $\ell$-term APs for this system (see definition in Section~\ref{SS:recurrence}).

\begin{definition}[Pointwise convergence along APs]
	We say that the sequence $a\colon \N\to\Z$ is 	good for {\em pointwise convergence  along  $\ell$-term APs for   $\ell$-step nilsystems} if 	 for all $k_1,\ldots, k_\ell\in\Z$,   $\ell$-step nilmanifolds $X=G/\Gamma$, $b\in G$, and
	$F_1,\ldots, F_\ell\in C(X)$, we have
	$$
	\lim_{N\to\infty}	\frac{1}{N}\sum_{n=1}^N  \, F_1(b^{k_1a(n)} x)\cdots 	 F_\ell(b^{k_\ell a(n)} x)=
	\lim_{N\to\infty}	\frac{1}{N}\sum_{n=1}^N  \, F_1(b^{k_1n} x)\cdots 	 F_\ell(b^{k_\ell n} x)
	$$
	for every $x\in X$.
\end{definition}
\begin{remark}
	In order to verify this property, we can assume that $b$ is an ergodic nilrotation. This is so because   the closure of the set
	$\{b^n\cdot e_X\colon  n\in \N\}$ is a subnilmanifold $Y$ of $X$ and $b$ acts ergodically in $Y$ (see Section~\ref{SS:nilbasic}).
\end{remark}

\subsection{Seminorm control}\label{SS:semi}
We next define some notions related to seminorm control of averages along $\ell$-term APs.
\begin{definition}[Seminorm control along $\ell$-term APs]
	We say that a sequence $a\colon \N\to \Z$ is {\em good for  seminorm control along $\ell$-term APs} for the system $(X, \CX, \mu,T)$,  if for all distinct  non-zero integers $k_1, \ldots, k_\ell\in\Z$, there exists some $s\in \N$ such that
	\begin{equation}\label{E:ki}
		\lim_{N\to\infty}	\frac{1}{N}\sum_{n=1}^N   \, T^{k_1a(n)}f_1  \cdots T^{k_\ell a(n)}f_\ell=0
	\end{equation}
	in $L^2(\mu)$ whenever $f_1,\ldots, f_\ell\in L^\infty(\mu)$ satisfy $\nnorm{f_j}_s=0$ for some $j\in[\ell]$. In this case, we also say that the averages \eqref{E:ki} {\em are good for seminorm control} for this system.

\end{definition}
\begin{remark}
	Using an ergodic decomposition argument, we see that if a sequence is good for degree $s$ seminorm control  along $\ell$-term APs for every ergodic system, then the same property holds for every system.
\end{remark}

There is an ample supply of sequences that are known to be good for seminorm control for the averages along  APs for  every  system, including non-constant polynomial sequences,  similar sequences evaluated at the prime numbers,  sequences arising from Hardy field functions of polynomial growth that grow faster than the logarithmic function.
We refer the reader to \cite[Section~3.2]{Fr11} for references and further examples.

An example of a sequence that fails to be good for seminorm control along $1$-step APs for some system can be constructed as follows. Take a system $(X,\CX, \mu,T)$ that is weakly mixing but not strongly mixing.
Then there exist a sequence $a(n)\to \infty$, 1-bounded real valued functions $f, g\in L^\infty(\mu)$ with $ \int g\, d\mu=0$, and $\varepsilon>0$, such that $\int f\cdot T^{a(n)}g\, d\mu\geq \varepsilon$ for every $n\in \N$.  Then the averages $\frac{1}{N}\sum_{n=1}^N\, T^{a(n)}g$ do not converge to 0 in $L^2(\mu)$.  At the same time, the assumptions that $T$ is weak mixing and $g$ has zero integral  imply that $\nnorm{g}_{s}  = 0$ for every $s\in\N$.
It follows that $(a(n))$ is not good for seminorm control for this system.

In Theorem~\ref{T:seminormdrop}, we  also consider the more general averages given by \eqref{E:average}. We  appropriately modify the previous definition to cover this more general setting.
\begin{definition}[Seminorm control] We say that the sequences $a_1,\ldots, a_\ell\colon \N\to \Z$ are
	\begin{enumerate}
		\item
		{\em good for degree $s$ seminorm control}
		for the system $(X, \CX, \mu,T),$ if
		$$
		\lim_{N\to \infty} \frac{1}{N}\sum_{n=1}^N   \, T^{a_1(n)}f_1\cdots T^{a_\ell(n)}f_\ell=0
		$$
		in $L^2(\mu)$ whenever $f_1,\ldots, f_\ell\in L^\infty(\mu)$ satisfy $\nnorm{f_j}_s=0$ for some $j\in[\ell]$.
		In this case, we also say that the seminorms $\nnorm{\cdot}_s$ {\em control the averages \eqref{E:average}} for this system.

		\item   {\em good for seminorm control} for the system $(X, \CX, \mu,T)$ if the previous property holds   for some $s\in \N$.
	\end{enumerate}
\end{definition}

We also define the crucial property with which we work in this paper.
\begin{definition} ($d$-step reduction)
	We call sequences $a_1, \ldots, a_\ell\colon \N\to\Z$ \emph{good for $d$-step reduction} for the system $(X, \CX, \mu, T)$ if for every $f_1, \ldots, f_\ell\in L^\infty(\CZ_d, \mu)$, the average \eqref{E:average} is 0 whenever $\nnorm{f_j}_{d}=0$ for some $j\in[\ell]$.
\end{definition}

Lastly, we extend the aforementioned properties to the more general setting of multivariable sequences and weighted averages along a F\o lner sequence.
\begin{definition}(More general averages)
	Let $k\in\N$, $(I_N)$ be a F\o lner sequence on $\Z^k$, and $(w_\un)$ be bounded complex weights on $\Z^k$. Then the sequences $a_1, \ldots, a_\ell\colon \Z^k\to\Z$ are	\begin{enumerate}
		\item
		{\em good for degree $s$ seminorm control
		for the system $(X, \CX, \mu,T)$ along $(I_N)$ with weights $(w_\un)$} if
		$$
		\lim_{N\to \infty} \frac{1}{|I_N|}\sum_{\un\in I_N}   \,w_\un\cdot T^{a_1(\un)}f_1\cdots T^{a_\ell(\un)}f_\ell=0
		$$
		in $L^2(\mu)$ whenever $f_1,\ldots, f_\ell\in L^\infty(\mu)$ satisfy $\nnorm{f_j}_s=0$ for some $j\in[\ell]$.
		In this case, we also say that the seminorms $\nnorm{\cdot}_s$ {\em control the averages \eqref{E:average}} for this system along $(I_N)$.
		\item   {\em good for seminorm control for the system $(X, \CX, \mu,T)$ along $(I_N)$ with weights $(w_\un)$} if the previous property holds for some $s\in \N$.
		\item \emph{good for $d$-step reduction for the system $(X, \CX, \mu, T)$ along $(I_N)$ with weights $(w_\un)$} if for every $f_1, \ldots, f_\ell\in L^\infty(\CZ_d, \mu)$, the average \eqref{E:average} is 0 whenever $\nnorm{f_j}_{d}=0$ for some $j\in[\ell]$.
	\end{enumerate}
\end{definition}

\subsection{Recurrence and divisibility}\label{SS:recurrence}
We record here a  notion related to recurrence properties along $\ell$-term APs
and a related notion regarding divisibility.

\begin{definition}[Multiple recurrence along APs]
	We say that a sequence $a\colon \N\to\Z$  {\em is good for recurrence along $\ell$-term  APs} if for every system $(X,\mathcal{X}, \mu,T)$, set $A\in \mathcal{X}$ with $\mu(A)>0$, and $k_1,\ldots, k_\ell\in \Z$,  we have
	$$
	\mu(A\cap T^{-k_1a(n)}A\cap \cdots \cap T^{-k_\ell a(n)}A)>0
	$$
	for some $n\in\N$.
\end{definition}
Using Furstenberg's correspondence principle~\cite{Fu77}, we deduce that if
$a\colon \N\to\Z$   is good for recurrence along $\ell$-term  APs, then  every set of integers with positive upper density contains patterns of the form $m,\; m+k_1a(n),\; \ldots,\; m+k_\ell a(n)$ for every $k_1,\ldots, k_\ell\in \Z$.

Examples of periodic systems show that if a sequence is good for single recurrence (i.e. recurrence along $1$-term APs), then its range should contain multiples of every positive integer. This motivates the following divisibility property that is part of the assumptions of our main recurrence result Theorem~\ref{T:APsRecurrence}.

\begin{definition}[Good divisibility property]
	We say that a sequence $a\colon \N\to\Z$ has {\em good divisibility properties} if for every $r\in \N$ the set $\{n\in\N\colon r|a(n)\}$ has positive  density.
\end{definition}
\begin{remark}
	If we relax the requirement on the set $\{n\in\N\colon r|a(n)\}$ to have positive upper density,  then we could still prove Theorem~\ref{T:APsRecurrence} below, but this would necessitate some (straightforward) technical modifications in our argument.
\end{remark}
Examples of sequences with good divisibility properties are  integer polynomials with zero constant terms (or any other intersective polynomial), the primes shifted by 1 (or -1), as well as integer parts of fractional powers and other Hardy field sequences that have polynomial growth and stay away from rational polynomials.  We refer the reader to \cite[Section~3.1]{Fr11} for references and further examples.

\subsection{Equidistribution on nilmanifolds}\label{SS:equi}
Lastly, we record some notions related to equidistribution properties on nilsystems,  see Section~\ref{SS:nilbasic} for related definitions and basic facts.

\begin{definition}[Equidistribution on nilmanifolds]
	We say that a sequence $(x_n)$ is equidistributed on a nilmanifold $X$ if
	$$
	\lim_{N\to\infty}\frac{1}{N}\sum_{n=1}^N F(x_n)=\int F\, dm_X
	$$ for every $F\in C(X)$.
\end{definition}

\begin{definition}[Good equidistribution properties]
	We say that a sequence $a\colon \N\to\Z$ is
	\begin{enumerate}
		\item 	{\em good for  $\ell$-step  equidistribution} if for every $\ell$-step nilmanifold
		$X=G/\Gamma$  and ergodic $b\in G$, the sequence
		$(b^{a(n)}\cdot e_X)$ is equidistributed  in $X$.
		
		\item  {\em good for  $\ell$-step  irrational equidistribution},
		if for every connected  $\ell$-step nilmanifold
		$X=G/\Gamma$  and ergodic $b\in G$, the sequence
		$(b^{a(n)}\cdot e_X)$ is equidistributed  in $X$.
	\end{enumerate}
\end{definition}
\begin{remarks}
	$\bullet$	In both cases,  we can deduce that 	$(b^{a(n)}\cdot x)$ is equidistributed  in $X$ for every $x\in X$. To see this, write $x=g\cdot e_X$ for some $g\in G$. Then  using Leibman's equidistribution criterion (see Section~\ref{SS:nilbasic}) we get that $g^{-1}bg$ is an  ergodic element   if $b$ is, and our assumption gives that the sequence	$b^{a(n)}\cdot x=g(g^{-1}bg)^{a(n)}\cdot e_X$ is equidistributed in $gX=X$.  See also Lemma~\ref{L:lstepequi} below for equivalent ways to verify $\ell$-step equidistribution.

	$\bullet$  If $X$ is a connected nilmanifold and $b$ is an ergodic rotation on $X$, then it is also totally ergodic, meaning that  $b^r$ is ergodic for every $r\in \N$. Conversely, if the nilmanifold $X$ admits a totally ergodic rotation, then it has to be connected. A proof of these basic properties can be found for example in  \cite[Chapter~11, Corollary~7]{HK18}.
\end{remarks}

The sequence $([n^c])$ where $c$ is a positive non-integer, is an example of a sequence that is good for $\ell$-step equidistribution for every $\ell\in \N$. Non-constant polynomial sequences typically fail to be good for $1$-step equidistribution due to congruence obstructions, but they are good for $\ell$-step equidistribution for every $\ell\geq 2$.  An example of a sequence that is good for $(\ell-1)$-step
equidistribution but not for $\ell$-step equidistribution is given in the fourth remark following Theorem~\ref{T:APsConvergence}.

We also remark that using Weyl's equidistribution criterion, it is easy to verify that the sequence $a\colon \N\to \Z$ is good for $1$-step equidistribution if and only if the sequence
$(a(n)\alpha)$ is equidistributed on $\T$ for every $\alpha\in (0,1)$,  and it is good for  $1$-step irrational equidistribution
if and only if the sequence  $(a(n)\alpha)$ is equidistributed on $\T$ for every irrational $\alpha\in (0,1)$.

\section{Main results} 
We now present the main results of the paper. 

\subsection{Mean convergence along arithmetic progressions}\label{SSS:APs1}
Given a system $(X, \CX, \mu,T)$, we are interested in studying the limiting behavior in $L^2(\mu)$ as $N\to \infty$ of the multiple ergodic averages with iterates taken along arithmetic progressions 
$$
\frac{1}{N}\sum_{n=1}^N   \, T^{a(n)}f_1\cdots T^{\ell a(n)}f_\ell,
$$
or the somewhat more general averages of the form
\begin{equation}\label{E:averagesGeneral}
	\frac{1}{N}\sum_{n=1}^N \, T^{k_1a(n)}f_1  \cdots T^{k_\ell a(n)}f_\ell
\end{equation}
where $k_1,\ldots, k_\ell\in \Z$ and
$f_1,\ldots, f_\ell\in L^\infty(\mu)$.

Our model is the case  $\ell=1$, in which case we are interested in  characterizing sequences $a\colon \N\to \Z$ for which
\begin{equation}\label{E:ell=1}
	\lim_{N\to\infty}\frac{1}{N}\sum_{n=1}^NT^{a(n)}f=\int f\, d\mu
\end{equation}
holds in $L^2(\mu)$ for all ergodic systems $(X, \CX, \mu,T)$ and $f\in L^\infty(\mu)$. By considering rotations on $\T$, it is straightforward to see that any such sequence has to be {\em good for equidistribution on the circle}, in the sense that $(a(n)t)$ is equidistributed in $\T$ for all $t\in (0,1)$. Conversely, this equidistribution property is sufficient for  \eqref{E:ell=1} to hold; this is a consequence of the spectral theorem for unitary operators, or more specifically the Herglotz theorem on positive definite sequences. In effect, we get the following well known result.
\begin{theorem*}
	The sequence $a\colon \N\to \Z$ satisfies  identity \eqref{E:ell=1} for all  ergodic systems  if and only if it is good for   equidistribution on the circle.
\end{theorem*}

Motivated by this result,  a slight variant of the  following conjecture was made in \cite[Conjecture~3]{FrJLW10} (see also \cite[Problem~3]{Fr11} for a related conjecture) and seeks to give a satisfactory answer to Question~\ref{Q: limits}. Throughout, we write APs in place of arithmetic progressions.

\begin{conjecture}[F.-Johnson-Lesigne-Wierdl~\cite{FrJLW10}]\label{Con:FJLW}
	Let  $a\colon \N\to \Z$ be a sequence and $\ell\in \N$. Then the  following properties  are equivalent:
	\begin{enumerate}
		
		\item $a$ is good for mean convergence along $\ell$-term APs for all   systems.
		
		
		\item  $a$ is good for $\ell$-step  equidistribution.
	\end{enumerate}
\end{conjecture}
\begin{remark}
	The implication $(i)\implies (ii)$ follows  quite  easily  from known results (see Proposition~\ref{P:ii-iii} below). The difficult implication is  $(ii)\implies (i)$, widely open even for  $\ell=2$.
\end{remark}

The next result verifies Conjecture~\ref{Con:FJLW}  for sequences that are  good for seminorm control of the averages \eqref{E:averagesGeneral} for all  ergodic   systems.
	
	

\begin{theorem}[Criteria for convergence along APs - general systems]\label{T:APsConvergence}
	Let  $a\colon \N\to \Z$ be a sequence  that is good for  seminorm control  along $\ell$-term APs  for all ergodic systems. Then the  following   properties are equivalent:
	\begin{enumerate}
		\item $a$ is good for mean convergence along $\ell$-term APs  for all  systems.
		
		\item $a$ is good for mean convergence along $\ell$-term APs
		for all  ergodic $\ell$-step nilsystems.
		
		\item  $a$ is good for $\ell$-step equidistribution.
	\end{enumerate}
\end{theorem}

\begin{remarks}		
	$\bullet$	The difficult implication is $(iii) \implies (i)$ (or $(ii) \implies (i)$). If we ask that the convergence properties hold for all $\ell\in \N$, then either implication follows from  the main structural result in \cite{HK05} (see Theorem~\ref{T:HK} below),  and we have nothing important to add. So our contribution is to prove these equivalences for every  fixed $\ell\in \N$, which is a really non-trivial matter.
		
	$\bullet$ Our argument shows that if a sequence is good for seminorm control along $\ell$-term APs for a fixed system $(X,\CX,\mu,T)$ and good for $\ell$-step equidistribution, then it is good for mean convergence along $\ell$-term APs for this system.
	
	$\bullet$ For $\ell\geq 2$, our standing assumption that the sequence $a\colon \N\to\Z$ is good for seminorm control  along $\ell$-term APs  is necessary for   property $(i)$ to hold. This follows from the fact that the averages   \eqref{E: AP average} are good for degree $\ell$ seminorm control (see  \cite[Theorem~8]{Lei05b}).

	$\bullet$ Condition~(iii) (or (ii)) cannot be improved in the following sense: for every $\ell\in \N$ there exists a sequence that is good for   $(\ell-1)$-step equidistribution  and good for  seminorm control along $\ell$-term APs for all systems, but fails to be  good for mean convergence along $\ell$-term APs for
	some ergodic system $(X, \CX, \mu,T)$ (which can be chosen to be an ergodic $\ell$-step nilsystem). An example is given by the sequence constructed by putting the elements of the set
	$
	\{ n\in\N\colon \norm{n^\ell \alpha} \in [1/4,1/2] \}
	$
	in increasing order, where $\alpha$ is irrational and $\norm{x}$ denotes the distance of $x$ from the nearest integer. The first good property follows from \cite[Proposition~4.2]{FrLW06} and the implication $(i)\implies (iii)$ in Proposition~\ref{P:ii-iii} below,  the second good property
	from  \cite[Proposition~4.2]{FrLW06}, and finally
	\cite[Proposition~3.2]{FrLW06}
	immediately implies  that this set is bad for  mean convergence along $\ell$-term APs.
\end{remarks}

We also prove the following  variant of Theorem~\ref{T:APsConvergence} that deals with totally ergodic systems and irrational equidistribution. The added value is that it applies to sequences that are not equidistributed in congruence classes (like polynomial sequences), and it will also be used to prove Theorem~\ref{T:APsRecurrence} below.
\begin{theorem}[Criteria for  convergence along $\ell$-term APs - totally ergodic systems]\label{T:APsConvergenceTE}
	Let $\ell\in\N$ and $a\colon \N\to \Z$ be a sequence  that is good for  seminorm control  along $\ell$-term APs  for all ergodic systems. Then the  following properties  are equivalent:
	\begin{enumerate}
		\item $a$ is
		good for mean convergence    along $\ell$-term APs for all  totally ergodic systems.
		
		\item  $a$  is good for  mean convergence along $\ell$-term APs for all  totally ergodic $\ell$-step nilsystems.
		\suspend{enumerate}
		Furthermore, both properties are implied by the following one:
		\resume{enumerate}
		\item   $a$  is good for   $\ell$-step irrational equidistribution.
	\end{enumerate}
\end{theorem}
\begin{remark} The difficult implication is $(ii)\implies (i)$ (or $(iii)\implies (i)$). It is probably the case that  the implication $(ii)\implies (iii)$ also holds, but we do not see how to adjust the
argument used to prove the implication $(ii)\implies (iii)$ in Theorem~\ref{T:APsConvergence} to  this case.
\end{remark}



We also state a related result about nilsystems; the  additional advantage in this case is that the stated equivalences hold for all sequences $a\colon \N\to \Z$. This follows from the fact that the degree $s+1$ Gowers-Host-Kra seminorm is a norm on an $s$-step nilsystem \cite[Chapter~9, Theorem~15]{HK18}, and hence every integer sequence is good for seminorm control along arithmetic progressions making Theorem~\ref{T:APsConvergence}  applicable for all integer sequences (see second remark following the theorem).

\begin{theorem}[Criteria for convergence along $\ell$-term APs - nilsystems]\label{T:nilAPs}
	Let  $a\colon \N\to \Z$ be a sequence and $\ell\in \N$. Then the  following properties are equivalent:
	\begin{enumerate}
		\item $a$ is good for mean convergence along $\ell$-term APs
		for all  nilsystems.

		\item $a$ is good for mean convergence along $\ell$-term APs
		for all  $\ell$-step nilsystems.
		
		\item $a$ is good for pointwise convergence along $\ell$-term APs
		for all $\ell$-step nilsystems.

		\item    $a$ is good for $\ell$-step equidistribution.
	\end{enumerate}
\end{theorem}
\begin{remark}
	The main novelty here is the implication $(ii)\implies (i)$.
	It is  also a somewhat curious fact that we can  infer pointwise convergence from mean convergence, i.e. the implication $(ii)\implies (iii)$.
\end{remark}

We can get similar results  without essential changes in the proofs for sequences in several variables  and averages along F\o lner sequences, see Section~\ref{SS:general}.

\subsection{Recurrence along arithmetic progressions}\label{SSS:APs2}
Motivated by Question \ref{Q: patterns}, we seek to give criteria that enable us to prove that a sequence $a\colon \N\to \Z$ is good for recurrence along $\ell$-term APs. Our model is the following  result from the 1970s that we rephrase using our terminology.
\begin{theorem*}[Kamae-Mend\`es France~\cite{KaMe78}]
	If the sequence  $a\colon \N\to \Z$  is
	good for $1$-step irrational equidistribution and
	has good divisibility properties, then it
	is good for single recurrence (i.e., recurrence along $1$-term APs).
\end{theorem*}

Inspired by this result, the following conjecture was made in  \cite[Conjecture~I]{FrLW09} (see also  \cite[Problem~4]{Fr11}) and seeks to give a satisfactory answer to Question~\ref{Q: patterns}.
\begin{conjecture}[F.-Lesigne-Wierdl~\cite{FrLW09}]\label{Con:FLW}
	If the sequence $a\colon \N\to \Z$ is
	good for $\ell$-step irrational equidistribution and
	has good divisibility properties, then it is
	good for recurrence along $\ell$-term APs.
\end{conjecture}
The next result verifies that Conjecture~\ref{Con:FLW}  holds for sequences that are  good for seminorm control along $\ell$-term APs.
\begin{theorem}[Criteria for recurrence along $\ell$-term APs]\label{T:APsRecurrence}
	If the sequence  $a\colon \N\to \Z$  is  good for seminorm control along $\ell$-term  APs for all ergodic systems, is good for  $\ell$-step irrational equidistribution, and
	has good divisibility properties, then it
	is good for recurrence along $\ell$-term APs.
\end{theorem}
	
	
\begin{remark}
It can be shown that for every $\ell\geq 2$, there exists a sequence that is good  for  $(\ell-1)$-step equidistribution, good for seminorm control for all systems,  and has good divisibility properties, but fails to be good for recurrence along $\ell$-term APs. An example is given by the sequence constructed by putting the elements of the set
$
\{ n\in\N\colon \norm{n^\ell \alpha} \in[1/4,1/2] \}
$
in increasing order where $\alpha$ is irrational.  The first two good properties follow from \cite[Proposition~4.2]{FrLW06} and the implication $(i)\implies (iii)$ in Proposition~\ref{P:ii-iii} below,  the third good property  follows easily using Weyl's equidistribution result, and finally \cite[Theorem A]{FrLW06}
shows that this sequence is bad for recurrence along $\ell$-term APs.
\end{remark}

\subsection{Mean convergence of cubic patterns}\label{SS:1.3} Another family of patterns that has attracted considerable interest is ``square'' patterns
$m,\; m+r,\; m+s,\; m+r+s$ and more general cubic patters of arbitrary degree. Our methods allow to prove the following result on square patterns with differences coming from two different integer sequences, that is, averages of the form
	\begin{equation}\label{E:squares}
	\frac{1}{N}\sum_{n=1}^N \, T^{a(n)}f_1\cdot T^{b(n)}f_2\cdot T^{a(n)+b(n)}f_3.
	\end{equation}

\begin{theorem}[Criteria for convergence of square averages]\label{T:squares}
Let $a,b\colon \N\to \Z$ be  sequences such that $a, b, a+b$ are good for seminorm control\footnote{Meaning that the averages \eqref{E:squares} are controlled by some Gowers-Host-Kra seminorm.} for all ergodic systems.
	Then the following properties are equivalent:
	\begin{enumerate}
		\item For all  systems   $(X, \CX, \mu,T)$  and  functions $f_1, f_2,f_3\in L^\infty(\mu)$, we have
		\begin{equation}\label{E:identitysquares}
			\lim_{N\to\infty}\frac{1}{N}\sum_{n=1}^N \, T^{a(n)}f_1\cdot T^{b(n)}f_2\cdot T^{a(n)+b(n)}f_3=\lim_{N\to\infty}\frac{1}{N^2}\sum_{r,s=1}^N \, T^{r}f_1\cdot T^{s}f_2\cdot T^{r+s}f_3.
		\end{equation}

		\item  Identity \eqref{E:identitysquares} holds  for all  ergodic $2$-step nilsystems.
	\end{enumerate}
	Furthermore,  property $(i)$ holds for all totally ergodic systems if and only if  property $(ii)$  holds for all totally ergodic $2$-step nilsystems.
\end{theorem}
\begin{remarks}
$\bullet$ We can generalize this to higher order cubic averages (by using $\ell$ sequences and putting $\ell$-step nilsystems in $(ii)$) without additional trouble. For the sake of simplicity, we will stick to the  case $\ell=2$.

$\bullet$ Condition~(ii) cannot be improved in the following sense: there exist sequences $a,b\colon \N\to \Z$ such that identity \eqref{E:identitysquares} holds for   all totally ergodic $1$-step nilsystems but fails for some $2$-step nilsystem. Take for example
$$
a(n):=n \cdot {\bf 1}_{[0,1/3]}(n^3\alpha), \quad b(n):=n^2 \cdot {\bf 1}_{[0,1/3]}(n^3\alpha),
$$
where $\alpha$ is irrational.
It is not hard to verify that the identity \eqref{E:identitysquares} holds  for all totally ergodic $1$-step nilsystems but fails for the system $T\colon \T^2\to\T^2$ defined by $T(x,y):=(x+\alpha,y+2x+\alpha)$ with the Haar measure $m_{\T^2}$. Indeed, if the identity were true, then a simple computation shows that the sequence $a(n)b(n)\alpha$ would be equidistributed on the circle, but it is not.

$\bullet$ If $a,b$ are two linearly independent polynomials in $\Z[n]$ with zero constant terms, then it was shown by the first author in \cite{Fr08} that identity \eqref{E:identitysquares} holds for all totally ergodic systems. As the previous example shows, however, this is not the case for more general ``linearly independent sequences'', for which it is necessary to verify equidistribution for $2$-step nilsystems.
\end{remarks}

\subsection{A stabilization property of the Host-Kra factors}
To each system $(X, \CX, \mu,T)$, we can assign the  Host-Kra factors $\CZ_0, \CZ_1, \ldots$, originally defined in \cite{HK05} for ergodic systems, some of their properties are discussed in Section~\ref{SS:seminorms}.
The following result was proved  by Host and Kra  in \cite[Chapter~16, Theorem~12]{HK18}.
\begin{theorem}[Stabilisation property for Host-Kra factors]\label{T:semistable1}
	Let $(X, \CX, \mu,T)$ be a system such that
	$\CZ_{d+1}=\CZ_{d}$ for some $d\in \N_0$.
	Then  $\CZ_{s}=\CZ_{d}$ 	for every $s\geq d$.
\end{theorem}
Host and Kra proved this result in \cite{HK18}  using their deep structural result stated in Theorem~\ref{T:HK} below. We give an alternative  proof that avoids such deep machinery and is based on Theorem~\ref{T:seminormdrop} below (or rather Theorem \ref{T:seminormdrop general}, its multivariable extension that can be proved in exactly the same way). Here is a sketch. For $s\in \N_0$, a function $f\in L^\infty(\mu)$ is $\CZ_s$-measurable if and only if it is orthogonal to all functions $g\in L^{\infty}(\mu)$ with $\nnorm{g}_{s+1}=0$.  So  it suffices to show the following property:

``If for some  $d\in \N$  the  seminorms $\nnorm{\cdot}_{d+1}$ and
$\nnorm{\cdot}_{d}$ have the same null space, then for every $s\geq d$  the  seminorms $\nnorm{\cdot}_{s}$ and
$\nnorm{\cdot}_{d}$  have the same null space.''
(The null space of a seminorm on $L^\infty(\mu)$ is the subset of $L^\infty(\mu)$ on which the seminorm vanishes.)

To see this, we let $s\geq d$ and
consider  the cubic averages
\begin{align*}
	\frac{1}{N^s}\sum_{1\leq n_1, \ldots, n_s\leq N} \prod_{\eps\in\{0,1\}^s}T^{\eps\cdot\un}f_\eps
\end{align*}
of degree $s$ that define the seminorms $\nnorm{\cdot}_{s}$ by identity  \eqref{E:seminorm single limit}. By the Gowers-Cauchy-Schwarz inequality (see for example \cite[Chapter~8, Theorem~13]{HK18}), these averages  are  good for degree $s$ seminorm control. Since by assumption the  seminorms $\nnorm{\cdot}_{d+1}$ and
$\nnorm{\cdot}_{d}$ have the same null space,
we deduce that if the degree $s$ cubic averages are good for degree $d+1$ seminorm control, then they are good for degree  $d$ seminorm control. Theorem~\ref{T:seminormdrop} then implies that the cubic averages of  degree $s$   are good for degree  $d$ seminorm control.  Hence, the null space of
the seminorm $\nnorm{\cdot}_s$ is contained in the null space of the seminorm $\nnorm{\cdot}_{d}$. The other inclusion is obvious since $s\geq d$, hence the two seminorms have the same null space.



\subsection{Mixing properties of nilsystems and applications}\label{SS:1.5}
The next result was proved in \cite{S69} for ergodic  nilsystems defined on nilmanifolds $X=G/\Gamma$ with $G$ connected and in \cite[Proposition~3.1]{HKM14} for general $2$-step nilsystems. It seems likely that the method in \cite{P70} can be modified to
extend this result to  general nilsystems (we thank Or Shalom and John Griesmer for bringing this article to our attention). In any case, we are going to use  Theorem~\ref{T:seminormdrop} to give an alternative proof of this result that covers the case of  general nilsystems.  We then use this result to prove in Theorem~\ref{T:mixinggeneral}  a  similar result for general systems.
\begin{theorem}[Mixing properties of nilsystems]\label{T:mixing}
	Let $(X,\CX,\mu,T)$ be an ergodic nilsystem and $f\in L^\infty(\mu)$. Then the following properties are equivalent:
	\begin{enumerate}
		\item$\nnorm{f}_2=0$, or equivalently, $\E(f|\CZ_1)=0$.
		
		\item For every strictly increasing sequence  $a\colon \N\to \N$, we have
		$$
		\lim_{N\to \infty}\norm{\frac{1}{N}\sum_{n=1}^N  T^{a(n)}f}_{L^2(\mu)}=0.
		$$
		
		\item For every $g\in L^\infty(\mu)$, we have
		$$
		\lim_{n\to\infty}\int g\cdot T^nf\, d\mu=0.
		$$
	\end{enumerate}
\end{theorem}
\begin{remark} The difficult implication is $(i)\implies (ii)$ (or $(i)\implies (iii)$).
	We could instead work with sequences   $a\colon \N\to \Z$  that have bounded multiplicity, in
	the sense  that $\sup_{k\in \Z}|\{n\in \N\colon a(n)=k\}|<\infty$.  The same holds for the other results in this subsection.
\end{remark}
\begin{proof}
	The implication $(iii)\implies (i)$ is straightforward. Moreover, the equivalence $(ii)\Leftrightarrow (iii)$ 	is pretty simple, see for example the proof of \cite[Theorem~1.1]{AS72} or \cite[Theorem~2.2]{BB86}.
	
	It remains to establish the implication $(i)\implies (ii)$.
	If the conclusion fails, then there exists 	an  ergodic  $s$-step nilsystem $(X,\CX,\mu,T)$, a function  $f\in L^\infty(\mu)$ with $\nnorm{f}_2=0$, a
	strictly increasing sequence $a\colon \N\to \N$, and $\varepsilon>0$,   such that
	\begin{equation}\label{E:epsilon}
		\limsup_{N\to \infty}\norm{\frac{1}{N}\sum_{n=1}^N  T^{a(n)}f}_{L^2(\mu)}\geq \varepsilon.
	\end{equation}
	Hence, in order to get a contradiction, 	it suffices to verify that every strictly increasing sequence  $a\colon \N\to \Z$ is good for degree $2$ seminorm control for the nilsystem $(X,\CX,\mu,T)$.
	
	Since any sequence is good for seminorm control for nilsystems (see the comment before Theorem~\ref{T:nilAPs}),
	by the remarks following Theorem~\ref{T:seminormdrop} we can assume that the function $f$ is $\CZ_2(\mu)$-measurable. Using Theorem~\ref{T:HK} (for nilsystems) and an approximation argument, we can assume that the system is an ergodic $2$-step nilsystem.
	In this case, by \cite[Proposition~3.1]{HKM14}, our assumption $\nnorm{f}_2=0$ implies  that $\lim_{n\to\infty}\int g\cdot T^nf\, d\mu=0$ for every $g\in L^\infty(\mu)$.
	 The implication $(iii)\implies (ii)$ then gives that
	$$
	\lim_{N\to \infty}\norm{\frac{1}{N}\sum_{n=1}^N  T^{a(n)}f}_{L^2(\mu)}=0,
	$$
	which contradicts \eqref{E:epsilon}.	
\end{proof}
The previous result can be used to deduce in a simple way that a set of Bohr recurrence is a set of topological recurrence for nilsystems (a different  self-contained proof of this fact was given in \cite[Section~4.3]{HKM15}), this was shown in  \cite[Section~4.2]{HKM15} for $2$-step nilsystems and the same argument applies for general nilsystems once Theorem~\ref{T:mixing} is known. We refer the reader to \cite{HKM15} for related definitions.

 Using the previous result, we deduce that if  the single term  averages
$\frac{1}{N}\sum_{n=1}^N  T^{a(n)}f$ are good for seminorm control for a specific system, then they are  good for degree $2$ seminorm control for that system.
\begin{theorem}[Seminorm control of single term averages]\label{T:mixinggeneral}
	Let  $a\colon \N\to \N$ be a strictly increasing sequence. Then the  following properties are equivalent:
	\begin{enumerate}
		\item $a$ is good for seminorm control
		for the system $(X,\CX,\mu,T)$.
		
		\item $a$ is good for degree $2$ seminorm control
		for the system $(X,\CX,\mu,T)$.
	\end{enumerate}
\end{theorem}
\begin{remarks}
    $\bullet$ Although Theorem \ref{T:mixinggeneral} concerns single term averages, its proof uses quite deep results like the Host-Kra structure theorem (Theorem \ref{T:HK}) and Theorem \ref{T:mixing} (which in turn relies on Theorem \ref{T:seminormdrop}). It would be interesting to find a more elementary proof of Theorem \ref{T:mixinggeneral} and/or its consequences Corollary \ref{C:convergence single averages} and Theorem \ref{T:singlereccurence}.

    $\bullet$ The example  given in the  fourth remark after Theorem~\ref{T:APsConvergence} shows that property $(ii)$ cannot be improved, in the sense that  there exist  sequences  (take $\ell=1$) and totally ergodic systems for which the  averages $\frac{1}{N}\sum_{n=1}^N  T^{a(n)}f$  are  good for degree $2$ but not    degree $1$ seminorm control.  On the other hand, the  $\ell=1$ case of Theorem~\ref{T:APsConvergence} (Theorem~\ref{T:APsConvergenceTE}) shows that if the sequence $a\colon \N\to \Z$ is such that $(a(n)t)$ is equidistributed on the circle for every (irrational) $t\in (0,1)$, then it is good for degree $1$ seminorm control for the averages  $\frac{1}{N}\sum_{n=1}^N  T^{a(n)}f$  for all (totally ergodic) systems.
\end{remarks}

\begin{proof}
	The implication $(ii)\implies (i)$ is trivial.
	
	We establish 	the implication $(i)\implies (ii)$. 	Using an ergodic decomposition argument, we can assume that the system is ergodic.
	Using Theorem~\ref{T:HK},
	we can assume that the system is an inverse limit of ergodic nilsystems,
	and a standard approximation argument allows us to further assume that the system is an ergodic nilsystem. In this  case, the result follows from the implication $(i)\implies (ii)$ of Theorem~\ref{T:mixing}.
\end{proof}
A related property for averages along arithmetic progressions can be found in Conjecture~\ref{Con:3} below. We move now to a result about mean convergence.
Given a system $(X,\CX,\mu,T)$, we let $\spec(T)$ be the set of $t\in [0,1)$ for which there exists non-zero $f\in L^2(\mu)$ such that $Tf=e^{2\pi i t}f$.
\begin{corollary}[Mean convergence for single term averages]\label{C:convergence single averages}
	Let  $a\colon \N\to \N$ be a strictly increasing sequence that is good for seminorm control for the system $(X,\CX,\mu,T)$.  Then the following properties are equivalent:
	\begin{enumerate}
		\item  The averages $\frac{1}{N}\sum_{n=1}^N  T^{a(n)}f$ converge in  $L^2(\mu)$ for  all $f\in L^2(\mu)$.
		
		\item The averages $\frac{1}{N}\sum_{n=1}^N  e^{2\pi ia(n)t}$ converge for all
		$t\in \spec(T)$.
	\end{enumerate}
\end{corollary}
\begin{remark}
	We stress that the spectral theorem cannot be  used to obtain this result, since the  convergence  property  in (ii)  applies  only to a countable number of $t\in [0,1)$.
\end{remark}
Lastly, we record a  result about single  recurrence.  If $(G,+)$ is a compact abelian group with a metric $d$,
we denote the \textit{distance from identity} of  $g\in G$ by $\norm{g}_G:=d(0,g)$.

\begin{definition}   A sequence $a\colon \N\to \Z$ is {\em good for averaged Bohr recurrence} if for every $d\in \N$, $\alpha \in \T^d$, and $\varepsilon >0$, the set
	 $\{n\in \N\colon  \norm{a(n)\alpha}_{\T^d}\leq \varepsilon\}$ has positive upper density.
	\end{definition}
\begin{remark}

By considering the rotation by $(1/r,\alpha)\in \T^{d+1}$, where	 $r,d\in \N$,
we deduce  that if the sequence $a\colon \N\to \Z$ is good for Bohr recurrence, then for every $\beta\in \Z_r\times \T^d$  and $\varepsilon>0$, the set $\{n\in \N\colon  \norm{a(n)\beta}_{\Z_r\times \T^d}\leq \varepsilon\}$ has positive upper density. Here, we think of $\Z_r$ as embedded in $\T$ in the natural way, hence the distance function on $\Z_r$ is the pullback of the distance function on $\T$, and the distance on $\Z_r\times \T^d$ can be taken to be the sum of distances on $\Z_r$ and $\T^d$.
\end{remark}

\begin{theorem}[Recurrence for single term averages]\label{T:singlereccurence}
	Suppose that the   strictly increasing sequence $a\colon \N\to \N$ is good for averaged Bohr recurrence and seminorm control
	for the system
	$(X,\CX,\mu,T)$. Then the sequence is good for single recurrence for this system, in particular, for every $A\in \CX$ with $\mu(A)>0$, we have
	\begin{equation}\label{E:AApos}
		\limsup_{N\to\infty}\frac{1}{N}\sum_{n=1}^N\mu(A\cap T^{-a(n)}A)>0.
	\end{equation}
\end{theorem}
\begin{remark}
	If   we replace upper  with lower density in the definition of the property ``good for averaged Bohr recurrence'', then we can replace the $\limsup$ with $\liminf$ in $\eqref{E:AApos}$.	\end{remark}
\begin{proof}
	Using an ergodic decomposition argument, we can assume that the system is ergodic.
	Theorem~\ref{T:HK} then gives that the system is an inverse limit of ergodic nilsystems,
	and using an approximation argument (as in \cite[Lemma~2.1]{BLL08} or \cite[Lemma~3.2]{FuK79}), we can assume that it is an ergodic nilsystem.
	We decompose
	$$
	{\bf 1}_A=f_1+f_2\quad \mathrm{for}\quad f_1:=\E(f|\CZ_1) \quad \mathrm{and}\quad f_2\, \bot\,  \CZ_1
	$$
	where $\CZ_1$ is the Kronecker factor of the system.
	By Theorem~\ref{T:mixing},  we have
	$$
	\lim_{N\to\infty}\norm{\frac{1}{N}\sum_{n=1}^N T^{a(n)}f_2}_{L^2(\mu)}=0,
	$$
	hence it suffices to show that
	\begin{equation}\label{E:Nk}
		\lim_{k\to\infty}\frac{1}{N_k}\sum_{n=1}^{N_k} \int f_1\cdot T^{a(n)}f_1\, d\mu>0
	\end{equation}
	for some sequence $N_k\to\infty$ that will be determined later.
	Since we assumed our system to be an ergodic nilsystem, the factor system on $\CZ_1$  is isomorphic to an ergodic  rotation on a compact abelian Lie group. To prove~\eqref{E:Nk}, we can therefore assume that $T$ is an ergodic rotation on  $G:=\Z_r\times \T^d$ for some $d,r\in \N_0$ with the Haar measure $\mu=m_G$. Since the  function
	$$
	F(t): = \int f_1(x)\cdot f_1(x+t)\, dm_G(x), \quad t\in G,
	$$  is continuous, there exists   $\varepsilon_0>0$ such that
	\begin{equation}\label{E:continuity}
		\sup_{\norm{t}_G\leq \veps_0}\Big|\int f_1(x)\cdot f_1(x+t)\, dm_G(x)- \int (f_1(x))^2\, dm_G(x) \Big|\leq\frac{1}{2} \int (f_1(x))^2\, dm_G(x).
	\end{equation}
	(Note that $\int (f_1(x))^2\, dm_G(x) \geq (m_G(A))^2 >0$.)
	Since $a\colon \N\to \N$ is good for averaged Bohr recurrence, we can choose  the sequence  $N_k$ so that the limit in \eqref{E:Nk} exists and
	$$
	\lim_{k\to \infty}  \frac{\big|\{n\in [N_k]\colon\  \norm{a(n)\alpha}_G\leq \veps_0\} \big|}{N_k}>0.
	$$
	We deduce from this and \eqref{E:continuity} that \eqref{E:Nk} holds. This completes the proof.
\end{proof}

\subsection{Key seminorm control result} The key new ingredient used in the proofs of the previous results is a new result that enables us to find the optimal degree of seminorm control in various cases. Our argument works without additional trouble in a more general setting than the one needed for averages along APs and cubic patterns, so we describe this more general case here.
Given a system $(X, \CX, \mu,T)$ and sequences $a_1,\ldots, a_\ell\colon \N\to\Z$, we consider the following averages
\begin{equation}\label{E:average}
	\frac{1}{N}\sum_{n=1}^N   \, T^{a_1(n)}f_1\cdots T^{a_\ell(n)}f_\ell
\end{equation}
for $f_1,\ldots, f_\ell\in L^\infty(\mu)$.


\begin{theorem}[Degree reduction property]\label{T:seminormdrop}
Let $\ell\in\N$ and suppose that the sequences $a_1,\ldots, a_\ell\colon \N\to\Z$ are good for seminorm control for the system $(X, \CX, \mu,T)$.  	Suppose also that for some $d\in \N$, degree $d+1$ seminorm control implies  degree $d$ seminorm control of the averages \eqref{E:average} for this system. Then the sequences $a_1, \ldots, a_\ell$ are good for  degree $d$ seminorm control for this system.
\end{theorem}
To put this differently, assume that for a fixed system, the sequences $a_1, \ldots, a_\ell$ are good for seminorm control, i.e. the averages \eqref{E:average} are controlled by Host-Kra seminorms of some (possibly very high) degree. In order  to prove degree $d$ seminorm control for some $d\in\N$, we can take for granted that we have degree $d+1$ seminorm control, or equivalently that all the functions are $\CZ_d$-measurable.
If degree $d$ control follows for all $\CZ_d$-measurable functions, we will say that the sequences $a_1, \ldots, a_\ell$ are \textit{good for $d$-step reduction}, a property that will be rigorously defined in the next section.

In various settings, this additional information simplifies matters quite a bit. For instance, Theorem~\ref{T:seminormdrop}  easily implies the  joint ergodicity result \cite[Theorem~1.1]{Fr21}, since verifying degree $1$ seminorm control assuming  degree $2$ seminorm control (in which case we can assume that the system is a rotation on a compact abelian Lie group) is an easy consequence of the given good equidistribution assumptions.

The proof of Theorem~\ref{T:seminormdrop} is inspired by the degree lowering approach taken in \cite{Fr21}, which in turn was inspired by finitary works of
Peluse~\cite{P19, P20} and Peluse and Prendiville~\cite{PP19}. However, there are severe technical difficulties in adapting the argument from \cite{Fr21} to our setting. The main challenge has been to find a suitable replacement for \cite[Proposition 3.2]{Fr21}, the inverse theorem for degree 2 seminorm that relates $\nnorm{f}_2$ to correlations of $f$ with eigenfunctions of $T$. Our first instinct for higher degree seminorms was to use the Host-Kra structure theorem (see Theorem~\ref{T:HK} below) to  relate $\nnorm{f}_d$ to correlations of $f$ with nilcharacters, which enjoy similar orthogonality properties to eigenfunctions. An argument of this form has recently been performed in the finitary work of Leng on finding bounds for subsets of finite fields lacking complexity 1 polynomial progressions \cite{Le22}. However, various technical obstacles encountered on the way pushed us to abandon this approach.

Instead, we use the trivial inverse theorem for degree $d$ seminorms, one that expresses $\nnorm{f}_d$ as a correlation of $f$ with its level $d$ dual function. We then carry out the degree lowering argument similarly to \cite{Fr21} but with level $d$ dual functions in place of eigenfunctions. The utility of eigenfunctions for the argument in \cite{Fr21} comes from their orthogonality properties. Dual functions lack orthogonality, but they possess strong cancellative properties that we exploit to our advantage in Proposition \ref{P:keyestimate}, the key novel technical input in this paper. Roughly speaking, Proposition \ref{P:keyestimate} states that if $\nnorm{f}_{d+s}=0$, then degree $d$ seminorms of certain dual functions of level $d+1$ associated with degree $s$ multiplicative derivatives of $f$ vanish on average.

In proving Theorem \ref{T:seminormdrop}, we also make crucial use of \cite[Proposition A.2]{FrKu22a}, a general functional analytic result that allows us to boost a qualitative seminorm control of a multilinear functional into a soft quantitative one.

\subsection{Generalizations to other averaging schemes}~\label{SS:general} Theorem \ref{T:seminormdrop} can be generalized in a straightforward way to weighted averages, sequences in several variables,  and averages along F\o lner sequences as follows.
\begin{theorem}\label{T:seminormdrop general}
	Let $k\in\N$, $(I_N)$ be a F\o lner sequence on $\Z^k$, $(w_\un)$ be bounded complex weights on $\Z^k$, and $a_1, \ldots, a_\ell\colon\Z^k\to\Z$ be sequences. Suppose that the sequences $a_1, \ldots, a_\ell$ 	are good for seminorm control for the system $(X, \CX, \mu,T)$ along $(I_N)$ with weights $(w_\un)$, i.e. the weighted averages
	\begin{equation}\label{E:weighted average}
		\frac{1}{|I_N|}\sum_{\un\in I_N}   w_\un\cdot T^{a_1(\un)}f_1\cdots T^{a_\ell(\un)}f_\ell
	\end{equation}
	are controlled by some Gowers-Host-Kra seminorm. Suppose also that for some $d\in \N$, degree $d+1$ seminorm control implies  degree $d$ seminorm control of the averages \eqref{E:weighted average} along $(I_N)$ with weights $(w_\un)$ for this system. Then $a_1, \ldots, a_\ell$ are good for  degree $d$ seminorm control for this system along $(I_N)$ with weights $(w_\un)$.
\end{theorem}

When used in place of Theorem \ref{T:seminormdrop}, Theorem \ref{T:seminormdrop general} allows us to extend Theorems \ref{T:APsConvergence}-\ref{T:squares} to averages along general F\o lner sequences on $\Z^k$ and sequences $a\colon \Z^k\to \Z$ in the natural way (it is trickier, however, to modify the statements in order to get interesting results for weighted averages).

\subsection{Further directions}
Theorems~\ref{T:APsConvergence} and \ref{T:APsConvergenceTE} give optimal seminorm control (of degree $\ell$) for averages along $\ell$-term APs for sequences that satisfy some good $\ell$-step equidistribution properties. In the absence of such good equidistribution properties, the example given in the  fourth remark after Theorem~\ref{T:APsConvergence} shows that  there exist  sequences for which the  averages along $\ell$-term APs are good for degree $\ell+1$ seminorm control but not   degree $\ell$ seminorm control. On the other hand, we have not been able to find  an example where seminorm control of degree higher than $\ell+1$ is needed for these averages. This motivates the next conjecture.
\begin{conjecture}\label{Con:3}
	Let $a\colon \N\to \Z$ be strictly increasing. If the  sequence  is good for seminorm control along $\ell$-term APs for the system $(X, \CX, \mu,T)$, then it is  good for  degree $\ell+1$ seminorm control along $\ell$-term APs for this system.
\end{conjecture}
Theorem~\ref{T:mixinggeneral} establishes Conjecture~\ref{Con:3} for $\ell=1$.
In order to establish  the conjecture  in full generality, using an ergodic decomposition argument we can assume that the system is ergodic.  Using the Host-Kra structure theorem (see Theorem~\ref{T:HK} below), it suffices to verify the conjecture  when the system is an ergodic nilsystem.  In this case,  a closely related problem is the following
higher order mixing property for nilsystems that would imply Conjecture~\ref{Con:3} for weak instead of strong convergence. It is also not hard to show that  Conjecture~\ref{Con:3} implies  Conjecture~\ref{Con:4}.
\begin{conjecture}\label{Con:4}
	Let $(X, \CX, \mu,T)$ be an ergodic   nilsystem and $k_1,\ldots, k_\ell\in \Z$ be non-zero and distinct.  Then
	$$
	\lim_{n\to\infty} \int f_0\cdot T^{k_1n}f_1\cdots T^{k_\ell n}f_\ell\, d\mu=0
	$$
	as long as $f_0,\ldots, f_\ell\in L^\infty(\mu)$ are such that $\nnorm{f_j}_{\ell+1}=0$ for some $j\in \{0,\ldots, \ell\}$.
\end{conjecture}
\begin{remark}
	A related problem is mentioned in  \cite{L15}  (see the remark after Proposition~1.1 there).
\end{remark}

Theorem~\ref{T:mixing} establishes Conjecture~\ref{Con:4} for $\ell=1$. Using Theorem~\ref{T:seminormdrop} and arguing as in the proof of  the implication $(i)\implies (ii)$ in Theorem~\ref{T:mixing}, we get that to prove Conjecture~\ref{Con:4}, it suffices to establish the conjecture under the additional assumption that the system is an $(\ell+1)$-step nilsystem.
If the $(\ell+1)$-step nilsystem is given by a unipotent affine transformation on a torus, then using direct computation, we can  verify the conjecture for all  $\ell\in \N$.   On the other hand,  the problem is open for general   nilsystems even when $\ell=2$ and $k_1=1,k_2=2$.


\section{Basic notation and background on seminorms}
In this section, we present various notions from ergodic theory together with some basic results.
\subsection{Basic notation}\label{SS:notation}
We start with explaining basic notation used throughout the paper.

The letters $\C, \R, \Z, \N, \N_0$ stand for the set of complex numbers, real numbers, integers, positive integers, and non-negative integers.    With $\T$, we denote the one dimensional torus, and we often identify it with $\R/\Z$ or  with $[0,1)$. We let $[N]:=\{1, \ldots, N\}$ for any $N\in\N$.  With $\Z[n]$, we denote the collection of polynomials with integer coefficients.


For an element $t\in \R$, we let $e(t):=e^{2\pi i t}$.

If $a\colon \N^s\to \C$ is a  bounded sequence for some $s\in \N$ and $A$ is a non-empty finite subset of $\N^s$,  we let
$$
\E_{n\in A}\,a(n):=\frac{1}{|A|}\sum_{n\in A}\, a(n).
$$


We often write $\eps\in\{0,1\}^s$ for a vector of 0s and 1s of length $s$. For $\eps\in\{0,1\}^s$ and $\uh, \uh'\in\Z^s$, we set
\begin{itemize}
	\item $\eps\cdot \uh:=\eps_1 h_1+\cdots+ \eps_s h_s$;
	\item $\abs{\uh} := |h_1|+\cdots+|h_s|$;
	\item $\uh^\eps := (h_1^{\eps_1}, \ldots, h_s^{\eps_s})$, where $h_j^0:=h_j$ and $h_j^1:=h_j'$ for $j=1,\ldots, s$.
\end{itemize}

We write $\CC z := \overline{z}$ for the complex conjugate of $z\in \C$.


\subsection{Ergodic seminorms}\label{SS:seminorms}
Given a system $(X, \CX, \mu,T)$, we will use the family of ergodic seminorms $\nnorm{\cdot}_s$,  also known as \emph{Gowers-Host-Kra seminorms}, which were originally introduced in  \cite{HK05} for ergodic systems. A detailed exposition of their basic properties can be found in  \cite[Chapter~8]{HK18}.
These seminorms are inductively defined for  $f\in L^\infty(\mu)$ as follows (for convenience, we also define $\nnorm{\cdot}_0$, which is   not a seminorm):
\begin{align}\label{E: degree 0 seminorm}
	\nnorm{f}_{0}:=\int f\, d\mu,
\end{align}
and for $s\in \N$, we let
\begin{equation}\label{E:seminorm1}
	\nnorm{f}_{s+1}^{2^{s+1}}:=\lim_{H\to\infty}\E_{h\in [H]} \nnorm{\Delta_{h}f}_s^{2^{s}},
\end{equation}
where
$$
\Delta_{h}f:=f\cdot T^h\overline{f}, \quad h\in \Z,
$$
is the \emph{multiplicative derivative of $f$}. If we deal with several transformations at a time, we may also denote the seminorm as $\nnorm{f}_{s, T}$ for clarity.
The limit can be shown to exist by successive applications of the mean ergodic theorem, and for $f\in L^\infty(\mu)$ and $s\in \N_0$, we have $\nnorm{f}_s\leq \nnorm{f}_{s+1}$ (see \cite{HK05} or  \cite[Chapter~8]{HK18}).
It follows immediately from the definition and the mean ergodic theorem that
$$
\nnorm{f}_{1}=\norm{\E(f|\CI)}_{L^2(\mu)},
$$
where $\CI:=\{f\in L^2(\mu)\colon Tf=f\}$, and
\begin{equation}\label{E:seminorm2}
	\nnorm{f}_s^{2^s}=\lim_{H_1\to\infty}\cdots \lim_{H_s\to\infty}\E_{h_1\in [H_1]}\cdots \E_{h_s\in [H_s]} \int \Delta_{ \uh}f\, d\mu,
\end{equation}
where  for $\uh=(h_1,\ldots, h_s)\in \Z^s$, we let
$$
\Delta_{\uh}f:=\Delta_{h_1}\cdots \Delta_{h_s}f=\prod_{\eps\in \{0,1\}^s}\mathcal{C}^{|\eps|}T^{\eps\cdot \uh}f
$$
be the \emph{multiplicative derivative of $f$ of degree $s$} with respect to $T$.

It can be shown that the previous limits exist and we can take any $s'\leq s$  of the  iterative limits to be simultaneous limits (i.e. average over $[H]^{s'}$ and let $H\to\infty$)
without changing the value of the limit in \eqref{E:seminorm2}. This was originally proved in \cite{HK05} using  the main structural result of \cite{HK05}; a more ``elementary'' proof  can be deduced from \cite[Lemma~1.12]{BL15} once the  convergence of the uniform Ces\`aro averages is known.
For $s':=s$, this gives the identity
\begin{equation}\label{E:seminorm single limit}
	\nnorm{f}_s^{2^s}=\lim_{H\to\infty}\E_{\uh\in [H]^s} \int \Delta_{\uh}f\, d\mu.
\end{equation}
Moreover,  for $1\leq s'\leq s$, we have
\begin{equation}\label{E:seminorm general single limit}
	\nnorm{f}_s^{2^{s}}=\lim_{H\to\infty}\E_{\uh\in [H]^{s-s'}} \nnorm{\Delta_{\uh}f}_{s'}^{2^{s'}}.
\end{equation}
For every $s\in\N_0$ and $d\geq 1$, we also have the inequality
\begin{align}\label{E: seminorm of powers}
    \nnorm{f}_{s, T}\leq \nnorm{f}_{s, T^d}
\end{align}
proved in \cite[Lemma 3.1]{FrKu22a}.


The seminorms are intimately connected with a certain family of factors of the system. Specifically, for every $s\in\N$ there exists a factor $\CZ_s\subseteq\CX$, known as the \emph{Host-Kra factor} of \emph{degree} $s$, with the property that
\begin{equation}\label{E:semifactor}
	\nnorm{f}_s = 0  \text{ if and only if } f \text{ is orthogonal to } \CZ_{s-1}.
\end{equation}
Equivalently, $\nnorm{\cdot}_s$ defines a norm on the space $L^2(\CZ_{s-1})$ \cite[Chapter~9, Theorem~15]{HK18}.

Lastly, we record the following simple inequality relating the seminorms of $f$ (with respect to $T$) to the seminorms of $f\otimes\overline{f}$ (with regards to $T\times T$ and $\mu\times \mu$) for all $s\in\N_0$:
\begin{align}\label{E:T vs TxT}
	\nnorm{f\otimes\overline{f}}_{s}\leq \nnorm{f}_{s+1}^2.
\end{align}
\subsection{Dual functions and sequences}\label{SS:dual}
Let $s\in\N$ and $\{0,1\}^s_*:= \{0,1\}^s\setminus\{\underline{0}\}$. For a system $(X, \CX, \mu, T)$, a function $f\in L^\infty(\mu)$, and $\um\in\N^s$, we define
$$
\Delta_{\um}^*f:=\prod_{\epsilon\in \{0,1\}^s_*} \CC^{|\ueps|}T^{\epsilon\cdot \um}f
$$
and
\begin{equation}\label{E:Dual}
\mathcal{D}_s(f):= \lim_{M\to\infty}\E_{\um\in [M]^s}\,\Delta_{\um}^*f = \lim_{M\to\infty}\E_{\um\in [M]^s}\prod_{\ueps\in\{0,1\}^s_*}\CC^{|\ueps|}T^{\ueps\cdot\um}f,
\end{equation}
 where the limit exists in $L^2(\mu)$ by \cite{HK05} (or \cite[Chapter~8, Theorem~28]{HK18}).
We call $\CD_s(f)$ the \emph{dual function of $f$ of level} $s$.
The name is motivated by the identity
\begin{align}\label{dual identity}
	\nnorm{f}_s^{2^s} = \int f \cdot \CD_s(f)\, d\mu,
\end{align}
a consequence of which is that the span of dual functions of level $s$ is dense in $L^1(\CZ_{s-1})$.

Rewriting $\um = (\um', m_s)$ and observing that $\Delta_{\um}^*f = \Delta_{\um'}^*f\cdot T^{m_s}\Delta_{\um'}\overline{f}$, we can express the dual function $\CD_s(f)$ as
\begin{align*}
	\mathcal{D}_s(f) & = \lim_{M\to\infty}\E_{\um'\in [M]^{s-1}}\lim_{M_s\to\infty}\E_{m_s\in [M_s]} \,\Delta_{\um', m_s}^*f\\
	&= \lim_{M\to\infty}\E_{\um'\in [M]^{s-1}}\brac{\Delta_{\um'}^*f \cdot \lim_{M_s\to\infty}\E_{m_s\in [M_s]}\, T^{m_s}\Delta_{\um'}\overline{f}}
\end{align*}
where our use of iterated limits is justified as in the case of formula~\eqref{E:seminorm2}.
In particular, for ergodic $T$, we deduce the identity
\begin{align}\label{E: dual identity}
	\mathcal{D}_s(f) = \lim_{M\to\infty}\E_{\um'\in [M]^{s-1}}\, c_{\um'} \cdot \Delta_{\um'}^*f,
\end{align}
where for $\um'\in \N^{s-1}$, we let
\begin{align*}
	c_{\um'} := \int \Delta_{\um'}\overline{f}\, d\mu.
\end{align*}

\section{Preliminary results for the proof of Theorem~\ref{T:seminormdrop}}\label{S:lemmas}
The proof of Theorem \ref{T:seminormdrop} provided in the next section uses a number of various technical results, most of which have previously appeared in various forms in our earlier works \cite{Fr21, FrKu22a}. We gather them in this section.


In order  to  apply  Proposition~\ref{P:soft quantitative} later on, we need to know that our averages converge in the mean. The next lemma enables us to work under this additional property.
\begin{lemma}[Producing mean convergence {\cite[Corollary~5.5]{FrKu22a}}]\label{L: strong limit}
	Let $(X, \CX, \mu, T)$ be a system, $a_1, \ldots, a_\ell\colon\N\to\Z$ be sequences, and $N_k\to\infty$. Then there exists a subsequence $(N'_k)$ of $(N_k)$ such that for every $f_1, \ldots, f_\ell\in L^\infty(\mu)$, the averages
	\begin{align}\label{E: strong limit}
		A_K := \E_{k\in[K]}\E_{n\in[N'_k]}\, T^{a_1(n)}f_1\cdots T^{a_\ell(n)}f_\ell
	\end{align}
	converge in $L^2(\mu)$ as $K\to\infty$.
\end{lemma}
	

Similarly to our arguments from \cite{Fr21, FrKu22a}, we will replace one of the functions $f_m$ in the average by a more structured term $\tilde{f}_m$ that encodes information about the original average. This will be achieved using the result below.
\begin{proposition}[Introducing 	$\tilde{f}_m$  {\cite[Proposition 5.6]{FrKu22a}}]\label{P:dual replacement}
	Let $a_1,\ldots, a_\ell\colon \N\to \Z$ be sequences, $(X, \CX, \mu, T)$ be a system, and  $f_1,\ldots, f_\ell\in L^\infty(\mu)$ be $1$-bounded functions, such that
	\begin{equation}\label{E:positivea}
		\limsup_{N\to\infty}\norm{\E_{n\in[N]} \,  T^{a_1(n)}f_1\cdots T^{a_\ell(n)}f_\ell}_{L^2(\mu)}>0.
	\end{equation}
	Let also $m\in[\ell]$. There exists a sequence of integers $N_k\to\infty$  such that for
	\begin{equation}\label{E:gk}
	g_k:=\E_{n\in [N_k]}\,  T^{a_1(n)}f_1\cdots T^{a_\ell(n)}f_\ell, \quad k\in \N,
	\end{equation}
	 the following holds: If
	\begin{equation}\label{E:averaged}
		A_N:=  \E_{n\in [N]} \, T^{-a_m(n)}\overline{g}_k\cdot  \prod_{j\in [\ell], j\neq m}T^{a_j(n)-a_m(n)}\overline{f}_j
	\end{equation}
    for $N\in\N$, then $A_{N_k}$ converges weakly to some $1$-bounded  function, and if we set
	\begin{equation}\label{E:ftildeweak}
		\tilde{f}_m:=\lim_{k\to\infty}\, A_{N_k},
	\end{equation}
	for $m\in \N$, where the limit is a weak limit, then additionally for every subsequence $(N'_k)$ of $(N_k)$ we have
	\begin{equation}\label{E:ftilde}
		\tilde{f}_m=\lim_{K\to\infty}\E_{k\in [K]}\, A_{N'_k} \, \text{ strongly in } L^2(\mu)
	\end{equation}
	and
	\begin{equation}\label{E:notzero1}
		\limsup_{k\to \infty} \norm{\E_{n\in[N_k']} \,   T^{a_m(n)}\tilde{f}_m \cdot \prod_{j\in [\ell],j\neq m}T^{a_j(n)}f_j}_{L^2(\mu)}>0.
	\end{equation}
\end{proposition}

An application of Proposition~\ref{P:dual replacement} in the proof of Theorem~\ref{T:seminormdrop} will be followed by an application of the result below for $f:=\tilde{f}_m$ (a proof of the variant given can be easily extracted
from the argument used to prove \cite[Proposition 5.7~(iii)]{FrKu22a}).
\begin{proposition}[Dual-difference interchange]\label{dual-difference interchange}
	Let $s, d\in \N$, $(X, \CX, \mu,T)$ be a system,
	$(f_{n,k})_{n,k\in\N}\subseteq L^\infty(\mu)$ be 1-bounded, and $f\in L^\infty(\mu)$ be defined by
	\begin{align*}
		f:=\, \lim_{K\to\infty}\E_{k\in[K]}\E_{n\in [N_k]}\, f_{n,k}
	\end{align*}
	for some $N_k\to\infty$ where the limit is taken in $L^2(\mu)$.
	If $\nnorm{f}_{s+d}>0$, then
	\begin{align}\label{E: dd interchange}
		\liminf_{H\to\infty}\E_{\uh, \uh'\in [H]^s}\limsup_{K\to\infty} \E_{k\in[K]} \E_{n\in[N_k]} \int\Delta_{ \uh-\uh'}f_{n,k} \cdot f_{\uh, \uh'} \, d\mu >0
	\end{align}
	for
	\begin{align*}
		f_{\uh, \uh'} := T^{-|\uh|}\prod_{\eps\in\{0,1\}^s}C^{|\epsilon|}\CD_d(\Delta_{\uh^\eps}f).
	\end{align*}
\end{proposition}

We also need the following result from \cite[Lemma~3.4]{Fr21},
which states  that if $\Delta_\uh f$ is twisted by low-complexity functions, then the average of the product is still bounded by Gowers-Host-Kra seminorms.
\begin{lemma}[Removing low-complexity functions]\label{L:lower}
	Let $s\in\N$,  $(X,\CX, \mu, T)$ be a  system, and  $f\in L^\infty(\mu)$ be a function.
	For $j\in [s]$ and $\uh\in [H]^s$, let  $c_{j,\uh}\in L^\infty(\mu)$ be 1-bounded functions such that 
	the sequence of functions $\uh\mapsto c_{j,\uh}$ does not depend on the variable $h_j$.
	If
	\begin{align}\label{positive limit in L:lower}
		\limsup_{H\to\infty}  	\norm{\E_{\uh\in [H]^s}\, \prod_{j=1}^s c_{j,\uh}\cdot \Delta_{\uh}f}_{L^2(\mu)}>0,
	\end{align}
	then $\nnorm{f}_{s} > 0$.
\end{lemma}

The proposition below enables a transition between qualitative and soft quantitative results; it is a version of \cite[Proposition A.1]{FrKu22a}.
\begin{proposition}[Soft quantitative control]\label{P:soft quantitative}
	Let $\ell, s\in\N$, $m\in[\ell]$, $a_1, \ldots, a_\ell\colon \N\to\Z$ be sequences,  and $(X, \CX, \mu, T)$ be a system. Let $\CY_1, \ldots, \CY_\ell\subseteq\CX$ be sub-$\sigma$-algebras. Suppose that there exists a sequence $N_k\to\infty$ such that for all functions $f_j\in L^\infty(\CY_j, \mu)$, $j=1,\ldots, \ell$,  the averages
	\begin{equation}\label{E:averages}
		\E_{k\in[K]}\E_{n\in[N_k]}  \, T^{a_1(n)}f_1\cdots T^{a_\ell(n)}f_\ell
	\end{equation}
	converge in $L^2(\mu)$ as $K\to\infty$, and moreover their limit is 0 whenever $\nnorm{f_m}_s = 0$. Then for every	$\varepsilon>0$	there exists  $\delta>0$, such that if the functions $f_j\in L^\infty(\CY_j, \mu)$  are $1$-bounded and  $\nnorm{f_m}_s\leq \delta$, then
	$$
	\lim_{K\to\infty} \norm{\E_{k\in[K]}\E_{n\in[N_k]} \,  T^{a_1(n)}f_1\cdots T^{a_\ell(n)}f_\ell}_{L^2(\mu)}\leq \varepsilon.
	$$
\end{proposition}
\begin{remark} The important point in the previous result is that $\delta$ does not depend on the $1$-bounded functions $f_1,\ldots, f_\ell$.
\end{remark}

We also record the following standard trick that we shall use a few times, and which allows us to get rid of the absolute value and weights from an average at the cost of passing to the product space.
\begin{lemma}[Product trick]\label{L:product trick}
	Let $\ell\in\N$,   $(X, \CX, \mu, T)$ be a system, and  $f_1, \ldots, f_\ell\in L^\infty(\mu)$ be functions. Then
	\begin{enumerate}
\item 	 For all $k_1,\ldots, k_\ell\in \Z$, we have
	\begin{multline*}
		\abs{\int T^{k_1}f_1\cdots T^{k_\ell}f_\ell\, d\mu}^2=\int (T\times T)^{k_1}(f_1\otimes\overline{f_1})\cdots (T\times T)^{k_\ell}(f_\ell\otimes\overline{f_\ell})\, d(\mu\times \mu).
	\end{multline*}

\item
For all sequences $a_1,\ldots, a_\ell\colon \N\to \Z$, 1-bounded weights  $w\colon \N\to \C$, and $N\in \N$,  we have
\begin{multline*}
	\norm{\E_{n\in[N]}\, w_n \cdot  T^{a_1(n)}f_1\cdots T^{a_\ell(n)}f_\ell}^2_{L^2(\mu)}\leq  \\
	\norm{\E_{n\in [N]} \, (T\times T)^{a_1(n)}(f_1\otimes\overline{f_1})\cdots (T\times T)^{a_\ell(n)}(f_\ell\otimes\overline{f_\ell})}_{L^2(\mu\times\mu)}.
\end{multline*}
	\end{enumerate}
\end{lemma}
\begin{proof}
	The identity in (i) is obvious. To prove the estimate in (ii), we expand the square inside the integral on the left hand side, use the triangle inequality to eliminate the weights, and then use the Cauchy-Schwarz inequality and the identity in (i) to arrive at the right hand side.
\end{proof}

\section{The degree lowering argument - reduction to a key seminorm estimate}
\subsection{Initial maneuvers}
In order to prove Theorem \ref{T:seminormdrop}, it is convenient to rephrase it in the following inductive form. We remind the reader that the notion of  sequences good for $d$-step reduction is defined in Section~\ref{SS:semi}.

\begin{proposition}\label{P:Main}
	Let $d\in\N_0$,  $\ell\in \N$, and $a_1,\ldots, a_\ell\colon \N\to \Z$ be sequences that are good for seminorm control and $d$-step reduction for the system $(X, \CX, \mu, T)$.
	Then for every $m\in \{0,1,\ldots, \ell\}$, the following property holds:
	
	$(P_m)$ If $f_1,\ldots,f_\ell\in L^\infty(\mu)$ and the functions  $f_{m+1}, \ldots, f_\ell$ are $\CZ_d$-measurable, then
	\begin{equation}\label{E:zeroo}
		\lim_{N\to\infty}\E_{n\in[N]}\,   T^{a_1(n)}f_1\cdots T^{a_\ell(n)}f_\ell=0
	\end{equation}
	in  $L^2(\mu)$  if $\nnorm{f_j}_d=0$ for some $j\in [\ell]$.
\end{proposition}

In the base case $m=0$,  property $(P_0)$ is equivalent to being good for $d$-step reduction (see definition in Section~\ref{SS:semi}). For $m\geq 1$, we will prove $(P_m)$ assuming $(P_{m-1})$. The case $(P_\ell)$ is equivalent to Theorem \ref{T:seminormdrop}.


We deduce from Proposition \ref{P:dual replacement} the following result.
\begin{proposition}\label{P:fmreplace}
	Let $d\in\N_0, \ell\in \N,$ and $a_1,\ldots, a_\ell\colon \N\to \Z$ be sequences that are good for seminorm control
	for the system $(X, \CX, \mu, T)$. Then  there exists $s\in\N$ such that the following holds: If $m\in [\ell]$  and   \eqref{E:zeroo} fails for some $f_1,\ldots, f_\ell\in L^\infty(\mu)$,
	then
	\begin{equation}\label{E:notzero2}
		\nnorm{\tilde{f}_m}_{s}>0
	\end{equation}
	where $\tilde{f}_m$ is given by \eqref{E:ftildeweak} (as a weak limit)
	and  satisfies \eqref{E:ftilde}.
\end{proposition}

The goal then is to establish the following statement.

\begin{proposition}\label{P:degreelowering}
	Let $d\in\N_0$, $\ell\in \N$,  and $a_1,\ldots, a_\ell\colon \N\to \Z$ be sequences that are good for seminorm control and $d$-step reduction for the system $(X, \CX, \mu, T)$. Suppose  that the	property $(P_{m-1})$  of Proposition~\ref{P:Main} holds for some $m\in [\ell]$ and  $\tilde{f}_m$ is given by  \eqref{E:ftildeweak} and  satisfies \eqref{E:ftilde} where all functions involved are $1$-bounded and $f_{m+1}, \ldots, f_\ell$ are $\CZ_d$-measurable.  Lastly,    suppose that $\nnorm{\tilde{f}_m}_{s+d+1}>0$ for some $s\in\N_0$.  Then   $\nnorm{\tilde{f}_m}_{s+d}>0$.
\end{proposition}
Before proceeding to the proof of Proposition~\ref{P:degreelowering}, we show how it can be used to prove
Proposition~\ref{P:Main} (and hence Theorem~\ref{T:seminormdrop}).
\begin{proof}[Proof of Proposition~\ref{P:Main} assuming Proposition~\ref{P:degreelowering}]
	Property $(P_0)$ is a trivial consequence of the assumption that the sequences $a_1,\ldots, a_\ell$ are good for $d$-step reduction for the system. To prove $(P_m)$ for $m\geq 1$, we split into the cases $m=1$ and $m\geq 2$; we need to do this because we will invoke $(P_1)$ while proving $(P_m)$ for $m\geq 2$.
	
	\medskip
	
	\textbf{The case $m=1$.}
		To prove property $(P_1)$, we recall from \eqref{E:gk} that \begin{align*}
		\tilde{f}_1 = \lim_{k\to\infty}\E_{n\in[N_k]}\, T^{-a_1(n)}\overline{g_k}\cdot\prod_{j=2}^\ell T^{a_j(n)-a_1(n)}f_j,
	\end{align*}
	where
	\begin{align*}
		g_k = \E_{n\in[N_k]}\, T^{a_1(n)}f_1\cdots T^{a_\ell(n)}f_\ell, \quad k\in \N.
	\end{align*}
	Propositions \ref{P:fmreplace} and \ref{P:degreelowering} then imply that if \eqref{E:zeroo} fails, then $\nnorm{\tilde{f}_1}_d>0$,  from which it follows using \eqref{dual identity} that
	\begin{align*}
		\lim_{k\to\infty}\int \overline{g}_k \cdot \E_{n\in[N_k]}\,  T^{a_1(n)}\CD_{d}(\tilde{f}_1)\cdot \prod_{j=2}^\ell T^{a_j(n)}f_j\, d\mu>0.
	\end{align*}
	The key point now is that since the functions $f_2, \ldots, f_\ell$  and  $\CD_{d}(\tilde{f}_1)$ are $\CZ_d$-measurable, so is the function $$F_k:=\E_{n\in[N_k]}\,  T^{a_1(n)}\CD_{d}(\tilde{f}_1)\cdot \prod_{j=2}^\ell T^{a_\ell(n)}f_j.$$
	Expanding the definition of $g_k$, we thus see that
	\begin{align*}
		\lim_{k\to\infty} \E_{n\in[N_k]}\int T^{a_1(n)}f_1\cdots T^{a_j(n)}f_\ell\cdot \overline{F_k}\, d\mu >0.
	\end{align*}
	Since for $n,k\in\N$, the function
	\begin{align*}
		F_{n,k}:= T^{a_2(n)}f_2\cdots T^{a_\ell(n)}f_\ell\cdot \overline{F_k}
	\end{align*}
	is $\CZ_d$-measurable, we get from the definition of conditional expectation that
	\begin{align*}
		\lim_{k\to\infty} \E_{n\in[N_k]}\int T^{a_1(n)}\E(f_1|\CZ_d)\cdot T^{a_2(n)}f_2\cdots T^{a_\ell(n)}f_\ell\cdot \overline{F_k}\, d\mu >0.
	\end{align*}
	By the Cauchy-Schwarz inequality, we have
	\begin{align*}
		\limsup_{k\to\infty} \norm{\E_{n\in[N_k]}\, T^{a_1(n)}\E(f_1|\CZ_d)\cdot T^{a_2(n)}f_2\cdots T^{a_\ell(n)}f_\ell}_{L^2(\mu)} >0.
	\end{align*}
	All the functions inside the $L^2(\mu)$-norm are $\CZ_{d}$-measurable, therefore the fact that $a_1, \ldots, a_\ell$ are good for $d$-step reduction (which is equivalent to the $(P_0)$ property) yields $\nnorm{\E(f_1|\CZ_d)}_d>0$, and hence also $\nnorm{f_1}_d=\nnorm{\E(f_1|\CZ_d)}_d>0$ by \eqref{E:semifactor}. It also yields $\nnorm{f_j}_d>0$ for $j\in \{2,\ldots, \ell\}$, completing the proof of $(P_1)$.
	
	\medskip
	
	\textbf{The case $m\geq 2$.}
	We  fix $m\geq 2$  and assume that property $(P_{m-1})$ holds.	Our goal is to show that property $(P_m)$ also holds.

	Arguing by contradiction, suppose  that \eqref{E:zeroo} does not hold, or equivalently, that \eqref{E:positivea} holds for some $f_1, \ldots, f_\ell\in L^\infty(\mu)$ where $f_{m+1}, \ldots, f_\ell$ are $\CZ_d$-measurable. It suffices to show that then $\nnorm{f_j}_d>0$ for every $j\in [\ell]$.

	Since the sequences $a_1,\ldots, a_\ell$ are good for seminorm control, by combining  Proposition~\ref{P:fmreplace} and   \eqref{E:positivea}, we deduce  that if $\tilde{f}_m$ is as in \eqref{E:ftildeweak}, then
	$\nnorm{\tilde{f}_m}_s>0$ for some $s\in \N$. If $s\leq d$, then this immediately implies $\nnorm{\tilde{f}_m}_d>0$ by the monotonicity property of Gowers-Host-Kra seminorms. Otherwise we iterate the conclusion of Proposition~\ref{P:degreelowering} exactly $s-d$ times, getting $\nnorm{\tilde{f}_m}_d>0$. Together with \eqref{dual identity}, this implies that
	$$
	\int \tilde{f}_m\cdot \CD_{d}(\tilde{f}_m)\, d\mu>0.
	$$
	We use  the definition \eqref{E:ftildeweak} of $\tilde{f}_m$, compose with  $T^{a_m(n)}$,  and then apply the Cauchy-Schwarz inequality to deduce  that
	$$
	\limsup_{k\to\infty}\norm{\E_{n\in[N_k]} \,T^{a_m(n)}\CD_{d}(\tilde{f}_m) \prod_{j\in [\ell], j\neq m} T^{a_j(n)}f_j}_{L^2(\mu)}>0.
	$$
	Since the functions $\CD_{d}(\tilde{f}_m), f_{m+1}, \ldots, f_\ell$ are $\CZ_{d}$-measurable,  property $(P_{m-1})$ implies that  $\nnorm{f_j}_d>0$
	for all $j\neq m$.
	
	It remains to show that $\nnorm{f_m}_d>0$. To do this, we first note that the argument above implies the identity
	\begin{equation}\label{E:Krat}
		\lim_{N\to\infty}\norm{ \E_{n\in[N]}\, \prod_{j\in[\ell]} T^{a_j(n)}f_j -  \E_{n\in[N]}\,    T^{a_m(n)}f_m \prod_{j\in [\ell], j\neq m}  T^{a_j(n)}\E(f_j|\CZ_{d-1})}_{L^2(\mu)}=0.
	\end{equation}
	Hence, the statement $\nnorm{f_m}_d>0$
	follows by invoking the case $(P_1)$ after an appropriate permutation of the sequences $a_1, \ldots, a_\ell$. This completes the proof.
\end{proof}

We remark that the argument for $(P_1)$ would quickly break if instead we considered averages of commuting transformations, i.e. each $a_j$ would be an iterate of a different transformation $T_j$. This is the only point in the proof of Theorem \ref{T:seminormdrop} in which we make essential use of the fact that we are only dealing with one transformation.

\subsection{Proof of Proposition \ref{P:degreelowering}: reduction to a key seminorm estimate}\label{SS:degree lowering}
Having derived Proposition \ref{P:Main} from Proposition~\ref{P:degreelowering}, we move on to prove Proposition \ref{P:degreelowering}. We will show that Proposition \ref{P:degreelowering} can be reduced to a seminorm estimate whose proof will occupy the entirety of the next section.


Let $(N_k)$ be the sequence that defines the function $\tilde{f}_m$ in Proposition \ref{P:dual replacement}, and let $(N'_k)$ the subsequence of $(N_k)$ given in Lemma \ref{L: strong limit} along which the sequence
\begin{equation}\label{E:AK}
	A_K(f_1, \ldots, f_\ell) = \E_{k\in[K]}\E_{n\in[N'_k]}\, T^{a_1(n)}f_1\cdots T^{a_\ell(n)}f_\ell
\end{equation}
converges in $L^2(\mu)$ as $K\to \infty$ for all $f_1, \ldots, f_\ell\in L^\infty(\mu)$. By Proposition \ref{P:dual replacement}, the function $\tilde{f}_m$ is also equal to the   strong limit
\begin{align*}
	\tilde{f}_m=\lim_{K\to\infty}\E_{k\in[K]}\E_{n\in[N_k']}\, f_{n,k}
\end{align*}
for
\begin{equation}\label{E:fnk}
	f_{n,k}:=T^{-a_m(n)}\overline{g}_k\cdot  \prod_{j\in [\ell], j\neq m}T^{a_j(n)-a_m(n)}\overline{f}_j.
\end{equation}
Using the assumption $\nnorm{\tilde{f}_m}_{s+d+1}> 0$, we deduce from Proposition \ref{dual-difference interchange} that
\begin{align*}
	\liminf_{H\to\infty} \limsup_{K\to\infty}\E_{k\in [K]}\E_{n\in [N'_k]}
	\E_{\uh,\uh'\in  [H]^{s}}\,  \int\Delta_{ \uh-\uh'}f_{n,k}\cdot f_{m, \uh, \uh'}\; d\mu>0
\end{align*}
for the $1$-bounded functions
\begin{equation}\label{E:hh'}
	f_{j, \uh, \uh'}:=\begin{cases} \Delta_{ \uh-\uh'}f_j,\; &j\neq m;\\
		T^{-|\uh|}\prod\limits_{\eps\in \{0,1\}^{s}}\mathcal{C}^{|\eps|}\CD_{d+1}(\Delta_{\uh^\eps}\tilde{f}_m),\; &j=m.
	\end{cases}
\end{equation}
For each $j=m+1,\ldots, \ell$, the function $f_j$ is $\CZ_d$-measurable, hence so is $f_{j,\uh,\uh'}$. Likewise, $f_{m, \uh, \uh'}$ is $\CZ_d$-measurable since each function $\CD_{d+1}(\Delta_{\uh^\eps}\tilde{f}_m)$ is.

Using the previous facts, the form of $f_{n,k}$ from \eqref{E:fnk}, and the mean convergence of $A_K$, we deduce after composing with $T^{a_m(n)}$ and  applying  the Cauchy-Schwarz inequality  that
\begin{align*}
	\liminf_{H\to\infty}\E_{\uh,\uh'\in[H]^{s}}
	\lim_{K\to\infty}\norm{\E_{k\in[K]}\E_{n\in [N'_k]} \, \prod_{j\in [\ell]} T^{a_j(n)}f_{j,\uh,\uh'}}_{L^2(\mu)}>0.
\end{align*}
In particular, an elementary pigeonholing argument guarantees that there exist a subset $\Lambda\subset\N^{2s}$ of positive lower density and $\varepsilon_0>0$ such that
\begin{align*}
	\lim_{K\to\infty}\norm{\E_{k\in[K]}\E_{n\in [N'_k]} \, \prod_{j\in [\ell]} T^{a_j(n)}f_{j,\uh,\uh'}}_{L^2(\mu)}>\varepsilon_0
\end{align*}
for all $(\uh, \uh')\in \Lambda$.

Using the fact that the functions $f_{j, \uh, \uh'}$ are $\CZ_d$-measurable for $j=m, \ldots, \ell$, we invoke the induction hypothesis $(P_{m-1})$ to deduce that
\begin{align*}
\lim_{K\to\infty}\norm{\E_{k\in[K]}\E_{n\in [N_k']} \, \prod_{j\in [\ell]} T^{a_j(n)}f_{j,\uh,\uh'}}_{L^2(\mu)} = 0
\end{align*}
whenever $\nnorm{f_{m, \uh, \uh'}}_d=0$. We boost this statement using Proposition~\ref{P:soft quantitative} (applied for $\CY_j := \CX$ if $j\in[m-1]$ and $\CY_j := \CZ_d$ otherwise) to deduce that for every $\varepsilon>0$, there exists $\delta>0$ such that
\begin{align*}
\lim_{K\to\infty}\norm{\E_{k\in[K]}\E_{n\in [N_k']} \, \prod_{j\in [\ell]} T^{a_j(n)}f_{j,\uh,\uh'}}_{L^2(\mu)} >\varepsilon
\end{align*}
implies $\nnorm{f_{m,\uh, \uh'}}_d>\delta$. Note that the mean convergence property of the averages~\eqref{E:AK} was crucial  here in order to be able to apply Proposition~\ref{P:soft quantitative}.  Picking $\varepsilon = \varepsilon_0$ as above, we conclude that
\begin{align*}
\liminf_{H\to\infty}\E_{\uh,\uh'\in[H]^{s}}\nnorm{f_{m,\uh, \uh'}}_d>0.
\end{align*}
(Recall the functions $f_{m,\uh, \uh'}$ were defined in \eqref{E:hh'}.)
That this implies $\nnorm{\tilde{f}_m}_{s+d}>0$ is the content of Proposition~\ref{P:keyestimate} in the next section.

\section{Key seminorm estimate - concluding the proof of Theorem~\ref{T:seminormdrop}}
We now prove the following key seminorm estimate needed to complete the degree lowering argument and hence the proof of Theorem~\ref{T:seminormdrop}.
\begin{proposition}\label{P:keyestimate}
Let $(X, \CX, \mu,T)$ be a system and $f\in L^\infty(\mu)$ be such that
$\nnorm{f}_{d+s}=0$, for some $d,s\in \N_0$.  Then
$$
\lim_{H\to\infty}\E_{\uh,\uh'\in [H]^s} \, \nnorm{f_{\uh, \uh'}}_{d}=0
$$
where
$$
f_{\uh, \uh'}:=\prod_{\epsilon\in \{0,1\}^{s}}\CC^{|\eps|} \mathcal{D}_{d+1}(\Delta_{\uh^\epsilon}  f)
$$
for $\uh,\uh'\in \N^s$.
\end{proposition}
To get a better sense of this result, it is useful to consider some special cases.

 The case  $s=0$ and  $d\in\N_0$    states that if
$\nnorm{f}_d=0$, then $\nnorm{\mathcal{D}_{d+1}f}_d=0$; this may  seem odd at first since the conclusion involves more differencing than the assumption.
For $d=1$ we want to show that if $\int f\, d\mu=0$ and the system is ergodic, then  $\int\mathcal{D}_{2}f\, d\mu=0$, which can be verified by a direct computation. Indeed, the identity \eqref{E: dual identity} implies that
\begin{align*}
\int \CD_2f\, d\mu = \lim_{M\to\infty}\E_{m\in[M]}\, c_m\int T^m f\, d\mu,
\end{align*}
for some uniformly bounded constants $c_m$, and so the claim readily follows from the assumption. The statement is trickier to prove for $d\geq 2$, and a proof of this special case is given in Section~\ref{SS:d0}.
Note also that if we want to strengthen the conclusion to $\nnorm{\mathcal{D}_{d+1}f}_{d+1}=0$, then we are forced to assume that
 $\nnorm{f}_{d+1}=0$ (which implies $\mathcal{D}_{d+1}f=0$ in $L^2(\mu)$).

On the other hand, the case  $d=0$ and $s\in \N_0$ follows from Lemma \ref{L:lower},  stating that cubic averages of order $s$ twisted by low complexity weights are controlled by the seminorm $\nnorm{\cdot}_s$, so again here the choice of constants is optimal.

The proof of Proposition~\ref{P:keyestimate} is given in Subsection~\ref{SS:ds}. For a better understanding, we include two additional examples that are notationally lighter.

\subsection{The case $d=2$, $s=1$}
We will show that if $f\in L^\infty(\mu)$ satisfies $\nnorm{f}_3=0$, then
\begin{align}\label{E:ds21 claim}
\lim_{H\to\infty}\E_{h,h'\in [H]} \, \nnorm{{f_h}\cdot \overline{f_{h'}}}_{2}=0
\end{align}
where $f_h:= \mathcal{D}_3(\Delta_h f)$ for $h\in \N$.
Using an ergodic decomposition argument, it suffices to prove this when the system is ergodic, which we assume. The identity \eqref{E: dual identity} allows us to express the functions $f_h$ as
$$
f_h= \mathcal{D}_3(\Delta_h f)=\lim_{M\to\infty}\E_{m_1,m_2\in[M]}\, c_{m_1,m_2, h}\cdot \Delta^*_{m_1,m_2}\Delta_h f
$$
for the uniformly bounded constants $c_{m_1,m_2, h} := \int \Delta_{m_1, m_2, h} \overline{f}\, d\mu$.

The identity \eqref{E:seminorm general single limit} implies that the claimed identity \eqref{E:ds21 claim} is equivalent to
$$
\lim_{H\to\infty}\E_{h,h'\in[H]}\,  \lim_{N\to\infty}\E_{n\in[N]} \nnorm{\Delta_n(f_h\cdot \overline{f_{h'}})}^2_1=0.
$$
Plugging in the formulas for $f_h, f_{h'}$ and expanding, we see that it suffices to show that
\begin{multline*}
\lim_{H\to\infty}\E_{h,h'\in [H]}\, \lim_{N\to\infty}\E_{n\in[N]}  \nnorm{\lim_{M\to\infty}\E_{\substack{m_i,m_i'\in[M]}}\, C_{\um, h, h'}\cdot\\ \Delta^*_{m_{1},m_{2}}\Delta_h f \cdot  \overline{\Delta^*_{m_{3},m_{4}}\Delta_{h'} f}\cdot T^n(\overline{\Delta^*_{m_{1}',m_{2}'}\Delta_h f}	\cdot  \Delta^*_{m_{3}',m_{4}'}\Delta_{h'} f)}^2_1=0,
\end{multline*}
where the averages $\E_{m_i,m_i'\in[M]}$ are taken over all $m_i,m_i', i=1,\ldots, 4,$ and the constants
\begin{align*}
C_{\um, h, h'}:=c_{m_{1}, m_{2}, h}\,  \overline{c_{m_{3}, m_{4}, h'}\, c_{m_{1}', m_{2}', h}}\,  c_{m_{3}', m_{4}', h'}
\end{align*}
are uniformly bounded.
Since the system is ergodic, the ergodic theorem gives  $\nnorm{g}_1=\abs{\int g\, d\mu}$, and so  using the Cauchy-Schwarz inequality,  it suffices to show that
\begin{multline*}
\lim_{H\to\infty}\E_{h,h'\in [H]}\, \lim_{N\to\infty}\E_{n\in[N]}\,  \lim_{M\to\infty}\E_{m_i,m_i'\in[M]}\\ \Big|\int  \Delta^*_{m_1,m_2}\Delta_h f \cdot  \overline{\Delta^*_{m_3,m_4}\Delta_{h'} f}\cdot
T^n\big(\overline{\Delta^*_{m_1',m_2'}\Delta_h f}	\cdot  \Delta^*_{m_3',m_4'}\Delta_{h'} f\big)\, d\mu\Big|^2=0.
\end{multline*}
 Note that the limits over $M$ and $N$ can be shown to exist, but we can  also work with the limsups without trouble.
Using Lemma~\ref{L:product trick}(i) as well as the identities  \eqref{E:T vs TxT} and  $\Delta_h(f_1\cdot f_2)=\Delta_hf_1\cdot \Delta_hf_2$,  we observe that upon replacing $T$ by $T\times T$ and $f$ by $f\otimes \overline{f}$, it suffices to show the following: if $\nnorm{f}_2=0$, then
\begin{multline*}
\lim_{H\to\infty}\E_{h,h'\in [H]}\, \lim_{N\to\infty}\E_{n\in[N]}\,  \lim_{M\to\infty}\E_{m_i,m_i'\in[M]} \\
\int  \Delta^*_{m_1,m_2}\Delta_h f \cdot  \overline{\Delta^*_{m_3,m_4}\Delta_{h'} f}\cdot  T^n\big(\overline{\Delta^*_{m_1',m_2'}\Delta_h f}\cdot  \Delta^*_{m_3',m_4'}\Delta_{h'} f\big)\, d\mu=0.
\end{multline*}
Bringing the averages of $m_i, m_i'$ inside, we rewrite the equality above as
\begin{multline*}
\lim_{H\to\infty}\E_{h,h'\in [H]}\,  \lim_{N\to\infty}\E_{n\in[N]}\,  \int  \lim_{M\to\infty}\E_{m_1,m_2\in[M]}\,  \Delta^*_{m_1,m_2}\Delta_h f \cdot  \lim_{M\to\infty}\E_{m_3,m_4\in[M]}\, \\ \overline{\Delta^*_{m_3,m_4}\, \Delta_{h'} f} \cdot
T^n\Big(\lim_{M\to\infty}\E_{m_1',m_2'\in[M]}\,  \overline{\Delta^*_{m_1',m_2'}\Delta_{h} f}
\cdot  \lim_{M\to\infty}\E_{m_3',m_4'\in[M]} \, \Delta^*_{m_3',m_4'}\Delta_{h'} f\Big)\, d\mu=0.
\end{multline*}
Using the definition of dual functions of level 2 we conclude that it suffices to show that
$$
\lim_{H\to\infty}\E_{h,h'\in [H]}\, \lim_{N\to\infty}\E_{n\in[N]}\,   \int  \mathcal{D}_2(\Delta_h f) \cdot   \overline{\mathcal{D}_2(\Delta_{h'} f)}  \cdot \\
T^n\big(\overline{\mathcal{D}_2(\Delta_h f)}\cdot  \mathcal{D}_2(\Delta_{h'} f) \big)\, d\mu=0,
$$
or that
$$
\lim_{H\to\infty}\E_{h,h'\in [H]}\nnorm{\mathcal{D}_2(\Delta_h f) \cdot   \overline{\mathcal{D}_2(\Delta_{h'} f)}}^2_1   =0
$$
by the definition of the Host-Kra seminorm of degree 1.
We have now effectively reduced our problem  to a simpler one, as we replaced the seminorm  $\nnorm{\cdot}_3$ in our assumption with $\nnorm{\cdot}_2$ and in the conclusion, we replaced the seminorm
$\nnorm{\cdot}_2$ with $\nnorm{\cdot}_1$ and the dual function $\mathcal{D}_3(\Delta_h f)$
with $\mathcal{D}_2(\Delta_h f)$. Repeating this argument one more time gives us a statement that can be verified directly using the ergodic theorem.

\subsection{Identities and estimates involving  multiplicative derivatives}\label{SS:identities}
Before we derive another special case of Proposition \ref{P:keyestimate}, we state several tedious but important identities that will show up in the proof. First we note  the identity
\begin{align}\label{E:product-average}
\prod_{\epsilon\in \{0,1\}^d}\sum_{m\in I} a_{\epsilon,m}=\sum_{m_{\epsilon}\in I, \epsilon\in \{0,1\}^d}
\prod_{\epsilon\in \{0,1\}^d}a_{\epsilon,m_{\epsilon}}
\end{align}
that holds for arbitrary finite sets $I$ and allows us to commute products and averages.
For instance, when $d=1$, $I=\{0,1\}$, the identity takes the form $(a_{0,0}+a_{0,1})(a_{1,0}+a_{1,1})=\sum_{m_0,m_1\in \{0,1\}} a_{0,m_0}\cdot a_{1,m_1}$.
Since  for fixed
$\un\in \N^{d-1}$  we have $\Delta_{\un}f =\prod_{\epsilon\in \{0,1\}^{d-1}}\CC^{|\eps|}T^{\epsilon\cdot \un}f$,
the previous identity implies that
\begin{align}\label{E:dualaverage}
\Delta_{\un}\brac{\lim_{M\to\infty}\E_{\um\in [M]^{d}}f_\um}&= \prod_{\epsilon\in \{0,1\}^{d-1}} \CC^{|\eps|}T^{\epsilon\cdot \un}\brac{\lim_{M\to\infty}\E_{\um\in [M]^{d}}f_\um}\\ \notag
&= \lim_{M\to\infty}\E_{\um_{\epsilon}\in [M]^d,\epsilon\in \{0,1\}^{d-1}} \prod_{\epsilon\in \{0,1\}^{d-1}}\CC^{|\eps|}T^{\epsilon\cdot \un}f_{\um_{\epsilon}}.
\end{align}
Plugging the functions $f_{\um}:=c_\um \cdot \Delta^*_{\um}f$ with $c_\um := \int \Delta_{\um} \overline{f}\, d\mu$ into the identity above and recalling that for ergodic systems $(X,\CX,\mu,T)$ we have
\begin{equation}\label{E:Dd+1}
\mathcal{D}_{d+1}f=\lim_{M\to\infty}\E_{\um\in [M]^{d+1}}\,  \Delta^*_{\um}f =\lim_{M\to\infty}\E_{\um\in [M]^{d}}\, c_\um \cdot \Delta^*_{\um}f,
\end{equation}
we obtain
\begin{align*}
\Delta_{\un} ( \mathcal{D}_{d+1}f)
&= \Delta_{\un} \brac{\lim_{M\to\infty}\E_{\um\in [M]^{d}}\, c_\um \cdot \Delta^*_{\um}f}\\
&=  \lim_{M\to\infty}\E_{\um_{\epsilon}\in [M]^d,\epsilon\in \{0,1\}^{d-1}} \prod_{\epsilon\in \{0,1\}^{d-1}}\CC^{|\eps|} T^{\epsilon\cdot \un}(c_{\um_{\epsilon}} \cdot \Delta^*_{\um_{\epsilon}}f)\\
&=\lim_{M\to\infty}\E_{\um_{\epsilon}\in [M]^d,\epsilon\in \{0,1\}^{d-1}}\, C_{(\um_\eps)}\cdot  \prod_{\epsilon\in \{0,1\}^{d-1}}\CC^{|\eps|} T^{\epsilon\cdot \un}( \Delta^*_{\um_{\epsilon}}f)
\end{align*}
for  $C_{(\um_\eps)} = \prod_{\eps\in\{0,1\}^{d-1}} \CC^{|\eps|}c_{\um_\eps}$.
Integrating over $X$, and applying the triangle and the Cauchy-Schwarz  inequalities gives the following estimate.
\begin{lemma}\label{L: deleting constants}
Let $d\in\N$, $\un\in \N^{d-1}$, $(X, \CX, \mu, T)$ be an ergodic system, and $f\in L^\infty(\mu)$ be 1-bounded. Then
\begin{align*}
	\abs{\int \Delta_{\un} ( \mathcal{D}_{d+1}f)\, d\mu}^2\leq \lim_{M\to\infty}\E_{\um_{\epsilon}\in [M]^d,\epsilon\in \{0,1\}^{d-1}}\abs{\int \prod_{\epsilon\in \{0,1\}^{d-1}}\CC^{|\eps|} T^{\epsilon\cdot \un}( \Delta^*_{\um_{\epsilon}}f)\, d\mu}^2.
\end{align*}
\end{lemma}
Lemma \ref{L: deleting constants} will play an important role in deriving Proposition \ref{P:keyestimate} when $s=0$ in Section \ref{SS:d0}. For the general case handled in Section \ref{SS:ds}, we need the following extension, whose proof follows  by 
essentially the same argument.
\begin{lemma}\label{L: deleting constants 2} 
Let $d\in \N, s\in\N_0$, $\un\in \N^{d-1}$, $(X, \CX, \mu, T)$ be an ergodic system, and $f_\eps'\in L^\infty(\mu)$ be 1-bounded for every $\eps'\in\{0,1\}^s$. Then 
\begin{multline*}
	\Big|\int \Delta_{\un} \Big(\prod_{\eps'\in\{0,1\}^s}\CC^{|\eps'|} \mathcal{D}_{d+1}f_{\eps'}\Big)\, d\mu\Big|^2\leq \\
	  \lim_{M\to\infty}\E_{\um_{\eps,\eps'}\in [M]^d,(\epsilon,\epsilon')\in \{0,1\}^{d-1}\times \{0,1\}^{s}}\Big| \int  \prod_{(\epsilon,\epsilon')\in \{0,1\}^{d-1}\times \{0,1\}^{s}}\CC^{|\eps|+|\eps'|} T^{\epsilon\cdot \un}(\Delta^*_{\um_{\eps,\eps'}}f_{\eps'})\, d\mu\Big|^2. 
	 \end{multline*}
\end{lemma}
 Note that the limits on the right hand side of the estimates in Lemma~\ref{L: deleting constants} and \ref{L: deleting constants 2} exist,  as they can be reinterpreted as  limits of cubic averages.
\subsection{ The case of general $d$ and $s=0$} \label{SS:d0}
We move to another special case, namely when there is no additional averaging over $\N^{2s}$ and differencing over $\N^s$.
We will induct on $d\in \N_0$ to prove the following statement:

\medskip

\textit{$(Q_{0, d})$:  For every system $(X, \CX, \mu, T)$ and function $f\in L^\infty(\mu)$, we have $\nnorm{\mathcal{D}_{d+1}f}_{d}=0$ whenever $\nnorm{f}_d = 0$.}

\begin{example}
	Before we proceed to the general case, we illustrate how a more computational proof works for $d=2$ and the affine system  $T(x_1, x_2)=(x_1+\alpha, x_2+2x_1+\alpha)$ on $X=\T^2$ for irrational $\alpha$. We give this example  just for reassurance, as this computational approach is not appropriate for extension.  We take our function to be   $f(x_1,x_2)=\sum_{l=1}^k\, c_l \, e(l x_2)$, where $c_1,\ldots, c_k\in \C$, noting that  $\nnorm{f}_2=0$.
	The identity $$e(l\cdot T^n(x_1,x_2))=e(l(x_2+2nx_1+n^2\alpha))$$ and a simple computation show that  $\mathcal{D}_{3}f$
	is a finite linear combination of functions of the form
	$$
	f_{l_1,l_2,l_3}(x_1,x_2)=\lim_{M\to\infty}\E_{m_1,m_2\in [M]} \, c_{m_1,m_2}\cdot e((l_1+l_2+l_3)x_2+(m_1(l_1+l_3)+m_2(l_2+l_3))x_1  )
	$$
	for some constants $c_{m_1,m_2}$ and  non-zero $l_1,l_2,l_3\in \Z$.
	Note that  if $l_1+l_2+l_3\neq 0$,  then 	$f_{l_1,l_2,l_3}$ depends non-trivially on the variable $x_2$, hence $\nnorm{f_{l_1,l_2,l_3}}_{2}=0$. On the other hand, if $l_1+l_2+l_3=0$, we cannot have
	$l_1+l_3=l_2+l_3=0$ (since one of the $l_1,l_2,l_3$ is non-zero),  hence the average defining 	the function $f_{l_1,l_2,l_3}(x_1,x_2)$  is zero a.e. (in fact for every $x_2$ as long as  $x_1$ is irrational). It follows that   $\nnorm{f_{l_1,l_2,l_3}}_{2}=0$ and hence $\nnorm{\CD_3f}_2=0$.
\end{example}

We move now to the proof of property $(Q_{0,d})$.

For $d=0$, the result follows from the identity
 $$\nnorm{\CD_1f}_0=\int \E(f|\CI(T))\, d\mu=\int f\, d\mu=
\nnorm{f}_0.
$$
So we assume that $d\geq 1$.  In proving property $(Q_{0,d})$ for a function $f$ and a system $(X, \CX, \mu, T)$, we will invoke property $(Q_{0, d-1})$ for $f\otimes \overline{f}$ and the system $(X\times X, \CX\otimes \CX, \mu\times \mu, T\times T)$. We recall from formula \eqref{E:T vs TxT} that $\nnorm{f}_d = 0$ implies $\nnorm{f\otimes \overline{f}}_{d-1} = 0$. This gives us a chain of implications
$$
\nnorm{f}_d =0\; \Rightarrow\; \nnorm{f\otimes\overline{f}}_{d-1} =0\; \Rightarrow\; \nnorm{\mathcal{D}_{d}(f\otimes \overline{f})}_{d-1}=0\;\Rightarrow\; \nnorm{\mathcal{D}_{d+1}(f)}_{d}=0,
$$
where the first implication follows from \eqref{E:T vs TxT} as explained above, the second implication is the induction hypothesis, and the third one we are about to prove.

Using an ergodic decomposition argument, it suffices to prove $\nnorm{\mathcal{D}_{d+1}(f)}_{d}=0$ whenever $\nnorm{f}_d = 0$ for ergodic systems. Using the iterative formula for Gowers-Host-Kra seminorms \eqref{E:seminorm1}, it is enough to show that
$$
\lim_{N\to\infty}\E_{\un\in [N]^{d-1}} \nnorm{\Delta_{\un} ( \mathcal{D}_{d+1}f)}^2_{1}=0,
$$
or, since the system is ergodic,  that
$$
\lim_{N\to\infty}\E_{\un\in [N]^{d-1}}  \Big|\int \Delta_{\un} ( \mathcal{D}_{d+1}f)\, d\mu\Big|^2=0.
$$
By Lemma \ref{L: deleting constants}, this will in return follow from
$$
\lim_{N\to\infty}\E_{\un\in [N]^{d-1}}\lim_{M\to\infty}\E_{\um_{\epsilon}\in [M]^d,\epsilon\in \{0,1\}^{d-1}} \Big|\int \prod_{\epsilon\in \{0,1\}^{d-1}}\CC^{|\eps|} T^{\epsilon\cdot \un}( \Delta^*_{\um_{\epsilon}} f)\, d\mu\Big|^2=0.
$$

Using Lemma~\ref{L:product trick}(i) and \eqref{E:T vs TxT}, we observe that the statement will follow if we show that
$$
\lim_{N\to\infty}\E_{\un\in [N]^{d-1}} \lim_{M\to\infty}\E_{\um_{\epsilon}\in [M]^d,\epsilon\in \{0,1\}^{d-1}}\int \prod_{\epsilon\in \{0,1\}^{d-1}}\CC^{|\eps|} (T\times T)^{\epsilon\cdot \un}( \Delta^*_{\um_{\epsilon}} (f\otimes \overline{f}))\, d(\mu\times \mu)=0
$$
whenever $\nnorm{f\otimes \overline{f}}_{d-1}=0$. Using  the identity \eqref{E:dualaverage}  for $f_\um:=\Delta_\um^*(f\otimes \overline{f})$ in order to bring the averages over  $\um_{\eps}$ inside $\Delta_\un$, we deduce that the identity above is equivalent to
$$
\lim_{N\to\infty}\E_{\un\in [N]^{d-1}}  \int \Delta_{\un} \brac{\lim_{M\to\infty}\E_{\um\in [M]^d}\,  \Delta^*_{\um} (f\otimes \overline{f})}\, d(\mu\times \mu)=0,
$$
which in turn holds  by \eqref{E:Dual} if and only if
$$
\lim_{N\to\infty}\E_{\un\in [N]^{d-1}} \int \Delta_{\un}\big(\mathcal{D}_{d}(f\otimes \overline{f})\big)\, d(\mu\times \mu)=0.
$$
Hence, using the iterative formula for Gowers-Host-Kra seminorms \eqref{E:seminorm1} and the formula \eqref{E: degree 0 seminorm}, we deduce that it suffices to show that
$$
\nnorm{ \mathcal{D}_{d}(f\otimes \overline{f})}_{d-1}=0
$$
where the seminorms are taken with respect to the product system.
This follows from the induction hypothesis.
\subsection{ The case of general $d$ and general  $s$} \label{SS:ds}
Finally, we proceed to tackle the general case of Proposition \ref{P:keyestimate}. Since the case $s=0$ was covered in Section~\ref{SS:d0} we can assume that $s\geq 1$. So  let $s\in\N$ be fixed. In analogy with the previous case, we will show the following property by induction on $d\in\N_0$.

\textit{$(Q_{s, d})$:  For every system $(X, \CX, \mu, T)$ and function $f\in L^\infty(\mu)$, we have
$$
\lim_{H\to\infty}\E_{\uh,\uh'\in [H]^s} \, \nnorm{f_{\uh,\uh'}}_d=0
$$
whenever $\nnorm{f}_{d+s}=0$, where
\begin{equation}\label{E:fhh'}
f_{\uh,\uh'}:=\prod_{\epsilon\in \{0,1\}^{s}}\CC^{|\eps|} \mathcal{D}_{d+1}(\Delta_{\uh^\epsilon}  f).
\end{equation}}

Like in Section \ref{SS:d0},  for $d\in \N$ and  fixed $s\in \N$, we will derive  property $(Q_{s,d})$ for a function $f$ and a system $(X, \CX, \mu, T)$ from  property $(Q_{s, d-1})$ for the function $f\otimes\overline{f}$ and the product system $(X\times X, \CX\otimes \CX, \mu\times \mu, T\times T)$.

By an ergodic decomposition argument, it suffices to prove  property $(Q_{s,d})$ for ergodic systems. Under this assumption,
 identity \eqref{dual identity} implies that
\begin{align}
\nonumber
\mathcal{D}_{d+1}(\Delta_{\uh^\epsilon}  f) &=\lim_{M\to\infty}\E_{\um\in [M]^{d+1}}\,  \Delta^*_{\um}(\Delta_{\uh^\epsilon}  f)\\
\label{E: f h,h'}
&= \lim_{M\to\infty}\E_{\um\in [M]^{d}}\, c_{\um,\uh,\uh'} \cdot \Delta^*_{\um}(\Delta_{\uh^\epsilon}  f)
\end{align}
for the uniformly bounded constants $c_{\um,\uh,\uh'}:=\int \Delta_{\um, \uh^\eps}\overline{f}\, d\mu$.

\medskip

{\bf Basis of induction.} For $d=0$, we will show that  property $(Q_{s,0})$ holds for every $s\in \N$.  It suffices to show that if $\nnorm{f}_s=0$, then
$$
\lim_{H\to\infty}\E_{\uh,\uh'\in [H]^s}\, \nnorm{f_{\uh,\uh'}}_0=0,
$$
where
\begin{align*}
f_{\uh,\uh'}:= \prod_{\epsilon\in \{0,1\}^{s}}\CC^{|\eps|} \mathcal{D}_{1}(\Delta_{\uh^\epsilon}  f) =\prod_{\epsilon\in \{0,1\}^{s}}\CC^{|\eps|} \int \Delta_{\uh^\epsilon}  f\, d\mu
\end{align*}
since the system  is ergodic. Hence, the claim will follow if we can show that
$$
\lim_{H\to\infty}\E_{\uh,\uh'\in [H]^s}\,\prod_{\epsilon\in \{0,1\}^{s}}\CC^{|\eps|} \int \Delta_{\uh^\epsilon}  f\, d\mu=0.
$$
In fact, we will show something stronger: if $\nnorm{f}_s=0$, then for every $\uh'\in \N^s$, we have
$$
\lim_{H\to\infty}\E_{\uh\in [H]^s}\prod_{\epsilon\in \{0,1\}^{s}}\CC^{|\eps|} \int \Delta_{\uh^\epsilon}f\, d\mu=0.
$$
Letting
$$
c_{\uh^\epsilon}:= \CC^{|\eps|}\int \Delta_{\uh^\epsilon}f\, d\mu
$$
for $\epsilon\in \{0,1\}^s$, we get that it suffices to show that
\begin{align}\label{E: d=0 final reduction}
\lim_{H\to\infty}\E_{\uh\in [H]^s}\prod_{\epsilon\in \{0,1\}^{s}_*}c_{\uh^\epsilon}\cdot  \int \Delta_{\uh}f\, d\mu=0.
\end{align}
But for $\epsilon \neq 0$, the sequence $\uh\mapsto c_{\uh^\epsilon}$ depends on at most $s-1$ of the variables $h_1,\ldots, h_s$. Hence, \eqref{E: d=0 final reduction} follows from Lemma \ref{L:lower} and the assumption that $\nnorm{f}_s = 0$.

\medskip

{\bf Induction step.} Suppose that
 property $(Q_{s,d-1})$ holds for some  $d\in \N$ and some fixed value of $s\in \N$. We will derive from this that
 property $(Q_{s,d})$ holds. Equivalently, using \eqref{E:seminorm general single limit} (with $d$ in place of $s$ and $s'=1$),  we want to show that
if $\nnorm{f}_{d+s}=0$, then
$$
\lim_{H\to\infty}\E_{\uh,\uh'\in [H]^s}\, \lim_{N\to\infty}\E_{\un\in [N]^{d-1}} \nnorm{\Delta_{\un} f_{\uh,\uh'}}^2_{1}=0,
$$
where $f_{\uh,\uh'}$ are given by \eqref{E:fhh'}.
Since the system is ergodic,  it suffices to show  that
$$
\lim_{H\to\infty}\E_{\uh,\uh'\in [H]^s} \, \lim_{N\to\infty}\E_{\un\in [N]^{d-1}} \Big|\int \Delta_{\un} f_{\uh,\uh'}\, d\mu\Big|^2=0.
$$
Plugging in the formula \eqref{E: f h,h'} for $f_{\uh, \uh'}$, using Lemma \ref{L: deleting constants 2} with $f_\eps := \Delta_{\uh^\epsilon} f$, $\eps \in \{0,1\}^s$, and averaging over $(\uh, \uh')\in\N^{2s}$, allows us to reduce the claim to the identity
\begin{multline*}
\lim_{H\to\infty}\E_{\uh,\uh'\in [H]^s}\lim_{N\to\infty}\E_{\un\in [N]^{d-1}}   \lim_{M\to\infty}\E_{\um_{\eps,\eps'}\in [M]^d,(\epsilon,\epsilon')\in \{0,1\}^{d-1}\times \{0,1\}^{s}}\\
\Big| \int  \prod_{(\epsilon,\epsilon')\in \{0,1\}^{d-1}\times \{0,1\}^{s}}\CC^{|\eps|+|\eps'|} T^{\epsilon\cdot \un}(\Delta^*_{\um_{\eps,\eps'}}\Delta_{\uh^{\epsilon'}} f)\, d\mu\Big|^2=0. 
\end{multline*}
Using Lemma \ref{L:product trick} and \eqref{E:T vs TxT}, it suffices to show that if $\nnorm{f\otimes \overline{f}}_{d+s-1}=0$, then
\begin{multline*}
\lim_{H\to\infty}\E_{\uh,\uh'\in [H]^s}\lim_{N\to\infty}\E_{\un\in [N]^{d-1}}
\lim_{M\to\infty}\E_{\um_{\eps,\eps'}\in [M]^d,(\epsilon,\epsilon')\in \{0,1\}^{d-1}\times \{0,1\}^{s}} \\\int  \prod_{(\epsilon,\epsilon')\in \{0,1\}^{d-1}\times \{0,1\}^{s}}\CC^{|\eps|+|\eps'|} T^{\epsilon\cdot \un}(\Delta^*_{\um_{\eps,\eps'}}\Delta_{\uh^{\epsilon'}} (f\otimes \overline{f}))\, d(\mu\times \mu)=0.
\end{multline*}

Using identity \eqref{E:product-average}
in order to bring the averages over $(\um_{\eps,\eps'})_{(\eps,\eps')}$ inside the product $\prod_{(\epsilon,\epsilon')\in \{0,1\}^{d-1}\times \{0,1\}^{s}}$, we deduce that the identity above is equivalent to
\begin{multline*}
\lim_{H\to\infty}\E_{\uh,\uh'\in [H]^s}
\lim_{N\to\infty}\E_{\un\in [N]^{d-1}}\\  \int \prod_{\eps'\in\{0,1\}^s} \CC^{|\eps'|} \Delta_{\un} \brac{\lim_{M\to\infty}\E_{\um\in [M]^d}\,  \Delta^*_{\um}\Delta_{\uh^{\epsilon'}} (f\otimes \overline{f})}\, d(\mu\times \mu)=0.
\end{multline*}
The definition of the dual function and the commutativity of $\prod_{\eps\in\{0,1\}^s}$ and $\Delta_\un$ allow us to write the preceding identity more compactly as
$$
\lim_{H\to\infty}\E_{\uh,\uh'\in [H]^s} \,  \lim_{N\to\infty}\E_{\un\in [N]^{d-1}} \int \Delta_{\un}\Big(\prod_{\eps'\in\{0,1\}^s} \CC^{|\eps'|}\mathcal{D}_{d}(\Delta_{\uh^{\epsilon'}} (f\otimes \overline{f}))\Big)\, d(\mu\times \mu)=0.
$$
Identity  \eqref{E:seminorm single limit}
then implies that the claim will follow from
$$
\lim_{H\to\infty}\E_{\uh,\uh'\in [H]^s} \,  \nnorm{\prod_{\eps'\in\{0,1\}^s} \CC^{|\eps'|} \mathcal{D}_{d}(\Delta_{\uh^{\eps'}} (f\otimes \overline{f}))}^{2^{d-1}}_{d-1}=0,
$$
where the seminorm is taken with respect to the product system. This  follows from  our induction hypothesis (for the product system), completing the proof.

\section{Preparation for the proofs of Theorems~\ref{T:APsConvergence}-\ref{T:squares}}

The remainder of the paper is devoted to proving Theorems~\ref{T:APsConvergence}-\ref{T:squares}. We collect here definitions and preliminary lemmas needed for the proofs of  Theorems~\ref{T:APsConvergence}-\ref{T:squares} while the actual proofs are postponed till the next section.
\subsection{Definitions and basic facts about nilsystems}\label{SS:nilbasic}
We start with basic definitions and  facts about nilsystems. Proofs of most of these facts  can be found in \cite{HK18, Lei05a}  and we give precise references below.

\medskip
\noindent {\em Nilmanifolds and nilsystems.}
An {\em $s$-step  nilmanifold} is a compact homogeneous space $X=G/\Gamma$ where $G$ is an $s$-step nilpotent Lie group and $\Gamma$ is a discrete cocompact subgroup.  For $b\in G$, the transformation $T\colon X\to X$ defined by $Tx=bx$ is called an {\em $s$-step  nilrotation} on $X$ and the system $(X,m_X,T)$, where $m_X$ is the projection of the Haar measure of $G$ on $X$,  is called an {\em $s$-step nilsystem.} With $G_0$ we denote the {\em connected component of the identity element in $G$}, which can be easily shown to be a clopen normal subgroup of $G$.

\medskip
\noindent {\em A standing assumption.} A nilmanifold has several representations and depending on our setting we may use a representation that has particular features. For instance, whenever we work with an ergodic nilsystem defined by a rotation $b\in G$, we may and will assume that $G$ is spanned by $G_0$ and $b$.

\medskip
\noindent {\em Unique ergodicity.} If $X=G/\Gamma$ is a nilmanifold and $b\in G$, $x\in X$, then the closure of the sequence $(b^n\cdot x)$ is (isomorphic to)
a nilmanifold  $Y$ and the action  of $b$ in $Y$ is uniquely ergodic  \cite[Chapter~11, Theorem~17]{HK18} or \cite[Section~2]{Lei05a}.

\medskip
\noindent {\em Ergodic nilsystems on connected nilmanifolds.} The nilmanifold  $X=G/\Gamma$ is connected if and only if $G=G_0\Gamma$ \cite[Chapter~10,  Lemma~11]{HK18}. Furthermore, a nilmanifold $X$ is connected if and only if it admits a totally ergodic nilsystem, and in this case any ergodic nilsystem on $X$ is totally  ergodic \cite[Chapter~11, Corollary~7]{HK18}.

\medskip
\noindent {\em Ergodic nilsystems on disconnected nilmanifolds.}
Let  $X$ be a  nilmanifold and $(X,m_X,T)$ be an ergodic nilsystem.
If $X_0$ is the connected component of $e_X$, then $X_0$ is a nilmanifold, and there exists $d\in \N$ such that
the sets $X_0,TX_0,\ldots, T^{d-1}X_0$ form a partition of $X$,
and $T^d$ leaves each of these sets invariant. Then each nilsystem $(T^iX_0,m_{T^iX_0},T^d)$, $i=0,\ldots, d-1$, is totally ergodic.
For a proof of these basic facts see  \cite[Chapter~11, Corollary~8]{HK18}.

\medskip
\noindent {\em The abelianization of a nilmanifold.} If $X$ is a nilmanifold, then
$Z:=X/[G,G]=G/([G,G]\cdot \Gamma)$ is the factor of $X$ called the {\em abelianization of $X$}, with the factor map $\pi_{X,Z}:X\to Z$ being the canonical projection. The space $Z$  is a compact abelian Lie group and so it can be identified with $H\times \T^k$ for a finite abelian group $H$ and $k\in \N_0$.
If $(X, \CX, \mu,T)$ is an ergodic  nilsystem, then necessarily $Z=\Z_d\times \T^k$ for some $d\in \N$ and $k\in \N_0$, and the system $(Z,m_Z,T)$ is the Kronecker factor of the system $(X, \CX, \mu,T)$.

\medskip
\noindent {\em Leibman's equidistribution criterion for linear sequences.}  The nilsystem $(X, \CX, \mu,T)$ is ergodic if and only if the system  $(Z,m_Z,T)$ is ergodic (see \cite[Chapter~11, Theorem~6]{HK18}
or \cite[Theorem~2.17]{Lei05a}). As a consequence, for $b\in G$, the sequence $(b^nx)$ is
equidistributed on $X$ if and only if the sequence $(b^n\cdot \pi_{X,Z}(x))$ is equidistributed/dense in $Z$.

\medskip
\noindent {\em Polynomial sequences.}  Let   $G$ be   a nilpotent Lie group and  $g(n)=b_1^{p_1(n)} \cdots b_k^{p_k(n)}$, where $b_1,\ldots, b_k\in G$ and $p_1,\ldots, p_k\in \Z[n]$. Let also $X=G/\Gamma$ be a nilmanifold.
We call any  sequence $(g(n))$ of the previous form a {\em polynomial sequence in } $G$ and
$ (g(n)\cdot e_X)$ a {\em polynomial sequence in} $X$.

\medskip
\noindent {\em Leibman's equidistribution criterion for  polynomial sequences.}  Let $X=G/\Gamma$ be a nilmanifold with $G$ connected. Then the  polynomial sequence $(g(n)\cdot e_X)$ is equidistributed in $X$ if and only if the sequence $(g(n)\cdot \pi_{X,Z}(x))$ is equidistributed/dense in $Z$ (see \cite[Chapter 14, Theorem~20]{HK18} or  \cite[Theorem~C]{Lei05a}).
We caution the reader that although for linear sequences this result continues to hold for disconnected $G$, it fails for general polynomial sequences when $G$ is not connected (in which case one has to test equidistribution on the larger nilmanifold $X/[G_0,G_0]$).

\medskip
\noindent {\em Factorization of linear sequences in $G_0$.}  Suppose that  $G=G_0\Gamma$ where $G_0$ is the connected component of $G$. Then for every  $b\in G$ there exist polynomial sequences $g_0(n)$ and $\gamma(n)$ of $G_0$ and $\Gamma$ respectively such that $b^n=g_0(n)\cdot \gamma(n)$ for every $n\in\N$.  This is a rather standard fact and we give a quick proof of this for completeness.

By \cite[Chapter~10, Proposition~25]{HK18}, we can assume that $G$ is a closed normal subgroup of a connected and simply connected nilpotent Lie group $\tilde{G}$ (we do this so that later on, \cite[Proposition~4.1]{BLL08} becomes applicable). Our assumption gives that $b=b_0\gamma$ for some $b_0\in G_0$ and $\gamma\in \Gamma$. Using this factorisation, the fact that  $G_0$ is a normal subgroup of $G$, as well as the inclusion $[G,G]\subset G_0$ (which follows from \cite[Chapter 10, Lemma 5]{HK18}), we easily deduce that $b^n=(b_0\gamma)^n=b_n\gamma^n$ for some $b_n\in G_0$. Hence,
$b_n=b^n\gamma^{-n}$ is a sequence in $G$, and therefore also in $\tilde{G}$.  By  \cite[Proposition~4.1]{BLL08},  the sequence  $b_n$ is a polynomial sequence in $G_0$, completing the proof.


\subsection{Nil-equidistribution results}
We record some convenient equivalent ways to verify $\ell$-step equidistribution.
\begin{lemma}\label{L:lstepequi}
	If  $a\colon \N\to \Z$ is a sequence, then the  following properties are equivalent:
	\begin{enumerate}
		\item  $a$ is good for  $\ell$-step equidistribution.
		
		\item  For every $\ell$-step nilmanifold $X=G/\Gamma$,  $b\in G$, and $F\in C(X)$,  we have
		$$
		\lim_{N\to\infty}\E_{n\in[N]}\, F(b^{a(n)}\cdot e_X)=\lim_{N\to\infty}\E_{n\in[N]} \, F(b^n\cdot e_X).
		$$
		
		\item  For every $\ell$-step nilmanifold $X=G/\Gamma$, $b\in G$,  $x\in X$, and $F\in C(X)$,  we have
		$$
		\lim_{N\to\infty}\E_{n\in[N]} \, F(b^{a(n)}\cdot x)=\lim_{N\to\infty}\E_{n\in[N]} \, F(b^n\cdot x).
		$$
	\end{enumerate}
\end{lemma}
\begin{proof}
	We prove that $(i)\implies (ii)$. It is known (see Section~\ref{SS:nilbasic}) that if  $Y$ is the closure
	of the set $\{b^n\cdot e_X\}$, then $Y$ is a subnilmanifold of $X$, $b$ acts ergodically on $Y$, and the sequence $(b^n\cdot e_X)$ is equidistributed in $Y$. Since $a\colon \N\to \Z$ is good for $\ell$-step equidistribution, we have
	$$
	\lim_{N\to\infty}\E_{n\in[N]}\, F(b^{a(n)}\cdot e_X)=\int F\, dm_Y=\lim_{N\to\infty}\E_{n\in[N]} \, F(b^n\cdot e_X),
	$$
	where the first identity follows from part $(i)$ and the second identity because
	$(b^n\cdot e_X)$ is equidistributed in $Y$.

	We prove that $(ii)\implies (iii)$. Let $x=g\cdot e_X$ for some $g\in G$.
	Then $$b^{a(n)}\cdot x=g(g^{-1}bg)^{a(n)}\cdot e_X$$ and
	\begin{align*}
		\lim_{N\to\infty}\E_{n\in[N]} \, F(b^{a(n)}\cdot x)&=
		\lim_{N\to\infty}\E_{n\in[N]} \, F(g(g^{-1}bg)^{a(n)}\cdot e_X)\\
		&=\lim_{N\to\infty}\E_{n\in[N]} \, F(g(g^{-1}bg)^{n}\cdot e_X)\\
		&=\lim_{N\to\infty}\E_{n\in[N]} \, F(b^n\cdot x),
	\end{align*}
	where we used part $(ii)$ for  the continuous function $x\mapsto F(gx)$ and the element $g^{-1}bg$ in place of $b$ in order to justify the second identity.

	We prove that $(iii)\implies (i)$. Let $X=G/\Gamma$ be an $\ell$-step nilmanifold,  $b$ be an ergodic element of $G$, and $F\in C(X)$. Then (see Section~\ref{SS:nilbasic})  we have $$\lim_{N\to\infty}\E_{n\in[N]} \, F(b^n\cdot e_X)=\int F\, dm_X,$$ and part $(iii)$ for $e_X$ in place of $x$  gives that
	$$
	\lim_{N\to\infty}\E_{n\in[N]}\, F(b^{a(n)}\cdot e_X)=\lim_{N\to\infty}\E_{n\in[N]} \, F(b^n\cdot e_X)=\int F\, dm_X.
	$$
	Hence, the sequence $(b^{a(n)}\cdot e_X)$ is equidistributed on $X$, which implies $(i)$.
\end{proof}
Next, we record some convenient equivalent ways to verify
mean convergence properties for nilsystems.
\begin{proposition}\label{P:ii-iii}
	If $a\colon \N\to \Z$ is a sequence and $\ell\in \N$, then the following properties  are equivalent:
	\begin{enumerate}
		\item  $a$ is good for  mean convergence along $\ell$-term APs for all  ergodic $\ell$-step nilsystems.
		
		\item  $a$ is good for pointwise  convergence along $\ell$-term APs for all  $\ell$-step nilsystems.
		
		\item  $a$ is good for   $\ell$-step  equidistribution.
	\end{enumerate}
\end{proposition}
\begin{remark}
	In the implication $(i)\implies (ii)$, we only have to assume in $(i)$ weak convergence and we can deduce that \eqref{E:X} folds for all $x_1,\ldots, x_\ell\in X$ in place of $x,\ldots, x$.
\end{remark}
\begin{proof}
	We prove that $(iii)\implies (ii)$. 	Suppose  that the sequence  $a\colon \N\to \Z$ is good for   $\ell$-step  equidistribution. Let $X=G/\Gamma$ be an $\ell$-step nilmanifold, $b\in G$, and $k_1,\ldots, k_\ell\in \Z$.  As explained in Section~\ref{SS:nilbasic}, we can assume that $b$ is an ergodic nilrotation of $X$.
	It suffices to show that if $F_1,\ldots, F_\ell\in C(X)$, then for  every  $x\in X$ we have
	\begin{equation}\label{E:X}
		\lim_{N\to\infty}\E_{n\in[N]} \, F_1(b^{k_1a(n)} x)\cdots 	 F(b^{k_\ell a(n)} x)=
		\lim_{N\to\infty}\E_{n\in[N]} \, F_1(b^{k_1n} x)\cdots 	 F(b^{k_\ell n} x).
	\end{equation}
	We let $\tilde{X}:=G^\ell/\Gamma^\ell$, which is an $\ell$-step nilmanifold, and  $\tilde{b}:=(b^{k_1},\ldots,b^{k_\ell})\in G^\ell$.
	Then  property $(iii)$ combined with  Lemma~\ref{L:lstepequi} (we use the implication $(i)\implies (iii)$ there), gives that for  every $\tilde{F}\in C(\tilde{X})$ and $\tilde{x}\in \tilde{X}$, we have
	\begin{equation}\label{E:tildeX}
		\lim_{N\to\infty}\E_{n\in[N]} \, \tilde{F}(\tilde{b}^{a(n)}\tilde{x})=\lim_{N\to\infty}\E_{n\in[N]} \, F(\tilde{b}^n\tilde{x}).
	\end{equation}
	If we apply this for  $\tilde{F}:=F_1 \otimes \cdots \otimes F_\ell$ and  $\tilde{x}:=(x,\ldots, x)\in \tilde{X}$, $x\in X$, we get that \eqref{E:X} holds for  every $x\in X$.
	
	The implication $(ii)\implies (i)$ follows from the bounded convergence theorem.
	
	We prove that $(i)\implies (iii)$.
	Let $X=G/\Gamma$ be an $\ell$-step nilmanifold, $b\in G$ be ergodic, and  $F\in C(X)$. For each $n\in\N$, we let
	$$
	\psi(n):= F(b^n\cdot e_X).
	$$
	It suffices to show that
	\begin{equation}\label{E:psi}
		\lim_{N\to\infty}\E_{n\in[N]} \, \psi(a(n)) =\lim_{N\to\infty}\E_{n\in[N]} \, \psi(n).
	\end{equation}
	It follows  from \cite[Proposition~2.4]{Fr15} that there exist $k_1,\ldots, k_\ell \in \N$   such that the following holds:  for every $\varepsilon>0$, there exist an $\ell$-step  nilsystem $(Y,m_Y,S)$ (which we may assume to be ergodic if we please) and functions $g_0,g_1,\ldots,g_\ell\in C(Y)$,
	such that  the sequence
	$$
	c_\varepsilon(n):=\int g_0\cdot S^{k_1n}g_1\cdots S^{k_\ell n}g_\ell\, dm_Y, \qquad n\in\N,
	$$
	satisfies
	\begin{equation}\label{E:approx}
		\norm{\psi-c_\varepsilon}_\infty\leq \varepsilon.
	\end{equation}
	Then property $(i)$  implies that both limits below exist and we have
	$$
	\lim_{N\to\infty}\E_{n\in[N]}\, c_\varepsilon(a(n))= \lim_{N\to\infty}\E_{n\in[N]}\, c_\varepsilon(n).
	$$
	Combining this with the fact that  \eqref{E:approx} holds for every $\varepsilon>0$, we get that \eqref{E:psi} holds.
\end{proof}
We will also need a variant of the implication $(iii)\implies (i)$ of the previous statement that works under the assumption of  $\ell$-step irrational equidistribution.
\begin{proposition}\label{P:te}
	If the sequence $a\colon \N\to \Z$   is good for   $\ell$-step irrational equidistribution, then
	it is good for  mean convergence along $\ell$-term APs for all totally ergodic $\ell$-step nilsystems.
\end{proposition}
\begin{remark}
	The converse implication is probably  also be true but it does not seem easy to prove. The problem is that  we cannot claim that the nilsystem $(Y,m_Y,S)$ constructed in  \cite[Proposition~2.4]{Fr15} is totally ergodic when $b$ is totally ergodic.
\end{remark}
\begin{proof}
	The argument is similar with the one used to establish the implications $(iii)\implies (ii)\implies (i)$ in Proposition~\ref{P:ii-iii}  but   we  need to make some non-trivial adjustments and use an additional result, so we give the details.
	
	Suppose  that the sequence  $a\colon \N\to \Z$ is good for   $\ell$-step  irrational equidistribution. Let $(X,m_X,T)$ be an   ergodic $\ell$-step nilsystem, where $X=G/\Gamma$ is a connected  $\ell$-step nilmanifold,  and $Tx=bx$  for some $b\in G$.
	As remarked in Section~\ref{SS:equi}, the nilsystem  $(X,m_X,T)$ is totally ergodic.
	It suffices to show that if $F_1,\ldots, F_\ell\in C(X)$, then for almost  every  $x\in X$ and all  $k_1,\ldots, k_\ell\in \Z$, we have
	\begin{equation}\label{E:X'}
		\lim_{N\to\infty}\E_{n\in[N]} \, F_1(b^{k_1a(n)} x)\cdots 	 F(b^{k_\ell a(n)} x)=
		\lim_{N\to\infty}\E_{n\in[N]} \, F_1(b^{k_1n} x)\cdots 	 F(b^{k_\ell n} x).
	\end{equation}
	We let $\tilde{X}:=G^\ell/\Gamma^\ell$, which is an $\ell$-step nilmanifold, and  $\tilde{b}:=(b^{k_1},\ldots,b^{k_\ell})\in G^\ell$. We also let
	$\tilde{X}_x$ be the closure of the set $\{\tilde{b}^n\cdot (x,\ldots, x)\colon n\in \Z\}$. It  is known that $\tilde{X}_x$ is  an $\ell$-step  sub-nilmanifold of $X^\ell$  for every $x\in X$  (see Section~\ref{SS:nilbasic}). Moreover, for almost every   $x\in X$, the action of $\tilde{b}$ on  $\tilde{X}_x$  is  totally ergodic (see \cite[Chapter~15, Lemma~9]{HK18} or \cite[Revised Theorem~7.1]{MR19}), or equivalently, $\tilde{X}_x$ is connected and $\tilde{b}$ acts ergodically on $\tilde{X}_x$ (see Section~\ref{SS:nilbasic}; importantly, this is not true for every $x\in X$).
	Since the sequence $a\colon \N\to \Z$ is  good for $\ell$-step irrational equidistribution,
	we deduce that for almost every $x\in X$, the following holds:
	for  every $\tilde{F}\in C(\tilde{X}_x)$ and $\tilde{x}\in \tilde{X}_x$, we have
	$$
	\lim_{N\to\infty}\E_{n\in[N]} \, \tilde{F}(\tilde{b}^{a(n)}\tilde{x})=\int \tilde{F} \, dm_{\tilde{X}_x}=\lim_{N\to\infty}\E_{n\in[N]} \, F(\tilde{b}^n\tilde{x}),
	$$
	where the first identity follows because  the sequence $a\colon \N\to \Z$ is  good for $\ell$-step irrational equidistribution (see also the first remark after the relevant definition) and the second because $\tilde{b}$ is an ergodic rotation of $\tilde{X}_x$ and hence uniquely ergodic  (see Section~\ref{SS:nilbasic}).
	If we apply this for  $\tilde{F}:=F_1 \otimes \cdots \otimes F_\ell$ and  $\tilde{x}:=(x,\ldots, x)\in \tilde{X}_x$, we get that \eqref{E:X'} holds for almost every $x\in X$.
\end{proof}
\section{Proof of Theorems~\ref{T:APsConvergence}-\ref{T:squares}}
We conclude the paper by deriving Theorems~\ref{T:APsConvergence}-\ref{T:squares} from Theorem~\ref{T:seminormdrop}.
While the proof of Theorem \ref{T:seminormdrop} made no use of the structure of the Host-Kra factors $\CZ_s$ other than the elementary equivalence \eqref{E:semifactor}, the proofs of Theorems~\ref{T:APsConvergence}-\ref{T:squares} will utilise in an essential way the main structural result from \cite{HK05} (see also  \cite[Chapter~16, Theorem~2]{HK18}), which we now state.
\begin{theorem}[Host-Kra structure theorem]\label{T:HK}
	Let $(X, \CX, \mu,T)$ be an ergodic system. Then for every $s\in \N$, the factor $\CZ_s$ is an inverse limit of $s$-step ergodic nilsystems.
\end{theorem}

\subsection{Proof of Theorem~\ref{T:APsConvergence}}
We first prove Theorem~\ref{T:APsConvergence} that we restate for convenience.
\begin{theorem} \label{T:APsConvergence'}
	Let  $\ell\in \N$   and $a\colon \N\to \Z$ be a sequence that is good for seminorm control along $\ell$-term APs for all ergodic systems. Then the following properties are equivalent:
	\begin{enumerate}
		\item $a$ is
		good for mean convergence  along $\ell$-term APs for all  systems.
		
		\item  $a$ is good for  mean convergence along $\ell$-term APs for all  ergodic $\ell$-step nilsystems.
		
		\item  $a$ is good for   $\ell$-step  equidistribution.
	\end{enumerate}
\end{theorem}
\begin{proof}
	Since the implication $(i)\implies (ii)$ is obvious and the equivalence $(ii)\Leftrightarrow (iii)$ was established in Proposition~\ref{P:ii-iii}, it suffices to show that $(ii)\implies (i)$.
	
	We first prove that the implication $(ii)\implies (i)$ holds for $\ell=1$.
	If we apply  property $(ii)$ for circle rotations, we get that
	$$
	\lim_{N\to\infty}\E_{n\in[N]} \, e(a(n)t)=0
	$$
	for all $t\in (0,1)$. Using the spectral theorem for unitary operators (or more specifically the Herglotz theorem on positive definite sequences), this easily implies that the sequence $a\colon \N\to \Z$ is good for $1$-convergence for all systems. (Note  that in this case, we did not have to assume good seminorm control, but we will for $\ell\geq 2$.)

	So it remains to prove that the implication $(ii)\implies (i)$ holds for  $\ell\geq 2$. Using an ergodic decomposition argument, we can assume that the system is ergodic. So let $(X, \CX, \mu,T)$ be an ergodic system and $k_1,\ldots,k_\ell\in \Z$. We can assume that the integers $k_1,\ldots, k_\ell$ are non-zero and  distinct.  Our first goal is to use Theorem~\ref{T:seminormdrop} in order to establish that the    averages \eqref{E:averagesGeneral} are good for degree $\ell$ seminorm control.
	Suppose that the   averages \eqref{E:averagesGeneral} are good for  degree $\ell+1$ seminorm control for this system. We claim that they are then good for degree $\ell$ seminorm control.  To see this, note first that our degree $\ell+1$  seminorm control  assumption implies that
	$$
	\lim_{N\to\infty}\brac{\E_{n\in[N]} \, T^{k_1a(n)}f_1  \cdots T^{k_\ell a(n)}f_\ell-
		\E_{n\in[N]} \, T^{k_1a(n)}\tilde{f}_1  \cdots T^{k_\ell a(n)}\tilde{f}_\ell}=0
	$$
	where the limit is taken in $L^2(\mu)$ and $\tilde{f}_j:=\E(f_j|\CZ_{\ell})$ for $j\in[\ell]$.
	So in order to prove degree $\ell$ seminorm control, it suffices to show that if $\nnorm{f_j}_\ell=0$ for some $j\in [\ell]$, then
	$$
	\lim_{N\to\infty}  \E_{n\in[N]} \, T^{k_1a(n)}\tilde{f}_1  \cdots T^{k_\ell a(n)}\tilde{f}_\ell=0
	$$
	where the limit is taken in $L^2(\mu)$.
	By Theorem~\ref{T:HK}, the factor $\CZ_\ell$ is an inverse limit of  $\ell$-step nilsystems, so using an approximation argument,
	we can assume that the system $(X, \CX, \mu,T)$ is an ergodic $\ell$-step nilsystem. In this case,  property $(ii)$
	gives that the last limit is equal to
	$$
	\lim_{N\to\infty}  \E_{n\in[N]} \, T^{k_1 n}\tilde{f}_1  \cdots T^{k_\ell n}\tilde{f}_\ell.
	$$
	But it is known \cite[Theorem~8]{Lei05b} that if $\ell\geq 2$, then  these averages are good for degree $\ell$ seminorm control, hence the claim follows.
	
	Thus, we have established that degree $\ell+1$ seminorm control for the averages \eqref{E:averagesGeneral} implies degree $\ell$ seminorm control and the identity \eqref{E:identityAPs}.  Theorem~\ref{T:seminormdrop} then implies that the averages \eqref{E:averagesGeneral} are good for degree $\ell$ seminorm control,  and as we just mentioned, we obtain the identity \eqref{E:identityAPs}.
	This completes the proof.
\end{proof}

\subsection{Proof of Theorem~\ref{T:APsConvergenceTE}}
We next  prove Theorem~\ref{T:APsConvergenceTE} that we  now restate.
\begin{theorem}\label{T:APsConvergenceTE'}
	Let $\ell\in \N$ $a\colon \N\to \Z$ be a sequence that is good for seminorm control along $\ell$-term APs for all ergodic systems. Then the following properties  are equivalent:
	\begin{enumerate}
		\item $a$ is
		good for mean convergence  along $\ell$-term APs for all  totally ergodic systems.
		
		\item  $a$ is good for  mean convergence along $\ell$-term APs for all  totally ergodic $\ell$-step nilsystems.
		\suspend{enumerate}
		Furthermore, both properties are implied by the following one:
		\resume{enumerate}
		\item   $a$ is good for   $\ell$-step irrational equidistribution.
	\end{enumerate}
\end{theorem}
\begin{proof}
	The implication $(i)\implies (ii)$ is trivial and the implication $(iii)\implies (ii)$ follows from Proposition~\ref{P:te}. So it remains to establish the implication $(ii)\implies (i)$. The  argument is similar to
	the one used to establish the implication $(ii)\implies (i)$ in Theorem~\ref{T:APsConvergence'}, so we omit the details.
\end{proof}

\subsection{Proof of Theorem~\ref{T:nilAPs}}
We restate  and prove Theorem~\ref{T:nilAPs}.
\begin{theorem}[Criteria for convergence along $\ell$-term APs - nilsystems]\label{T:nilAPs'}
	Let  $a\colon \N\to \Z$ be a sequence and $\ell\in \N$. Then the  following properties are equivalent:
	\begin{enumerate}
		\item $a$ is good for mean convergence along $\ell$-term APs
		for all  nilsystems.

		\item $a$ is good for mean convergence along $\ell$-term APs
		for all  $\ell$-step nilsystems.
		
		\item $a$ is good for pointwise convergence along $\ell$-term APs
		for all $\ell$-step nilsystems.

		\item    $a$ is good for $\ell$-step equidistribution.
	\end{enumerate}
\end{theorem}
\begin{proof}
	The equivalence of the last three properties was established in  Proposition~\ref{P:ii-iii}  	and  the implication $(i)\implies (ii)$ is trivial.
	
	So it remains to establish the implication $(ii) \implies (i)$.  Since  the seminorm $\nnorm{\cdot}_{s+1}$ is a norm on an $s$-step nilsystem  \cite[Chapter~12, Theorem~17]{HK18}, all sequences are  good for seminorm control along $\ell$-term APs for every nilsystem. Taking into account this property, the argument is identical to the one used to  establish the implication   $(ii)\implies (i)$ of Theorem~\ref{T:APsConvergence} (see previous subsection).
\end{proof}

\subsection{Proof of Theorem~\ref{T:squares}.}
We restate  and prove Theorem~\ref{T:squares}.
\begin{theorem}[Criteria for convergence of square averages]\label{T:squares'}
Let $a,b\colon \N\to \Z$ be  sequences such that $a, b, a+b$ are good for seminorm control for all ergodic systems.
	Then the following properties are equivalent:
	\begin{enumerate}
		\item For all  systems   $(X, \CX, \mu,T)$  and  functions $f_1, f_2,f_3\in L^\infty(\mu)$, we have
		\begin{equation}\label{E:identitysquares'}
			\lim_{N\to\infty}\frac{1}{N}\sum_{n=1}^N \, T^{a(n)}f_1\cdot T^{b(n)}f_2\cdot T^{a(n)+b(n)}f_3=\lim_{N\to\infty}\frac{1}{N^2}\sum_{r,s=1}^N \, T^{r}f_1\cdot T^{s}f_2\cdot T^{r+s}f_3.
		\end{equation}

		\item  Identity \eqref{E:identitysquares'} holds  for all  ergodic $2$-step nilsystems.
	\end{enumerate}
	Furthermore,  property $(i)$ holds for all totally ergodic systems if and only if  property $(ii)$  holds for all totally ergodic $2$-step nilsystems.
\end{theorem}
\begin{proof}	
	The implication $(i)\implies (ii)$ is obvious.
	
	We establish the implication $(ii)\implies (i)$. 	Using an ergodic decomposition argument we can assume that the system is ergodic.
	So lets fix  an ergodic system  $(X, \CX, \mu,T)$ and suppose that the   averages on the left hand side of  \eqref{E:identitysquares'} are good for degree $3$ seminorm control for this system. We claim that they are then good for degree $2$ seminorm control.  To see this, note first that our assumption implies that
	$$
	\lim_{N\to\infty}\Big( \E_{n\in[N]} \,  T^{a(n)}f_1\cdot T^{b(n)}f_2\cdot T^{a(n)+b(n)}f_3-
	\E_{n\in[N]} \,  T^{a(n)}\tilde{f}_1\cdot T^{b(n)}\tilde{f}_2\cdot T^{a(n)+b(n)}\tilde{f}_3\Big)=0
	$$
	where the limit is taken in $L^2(\mu)$ and $\tilde{f}_j:=\E(f_j|\CZ_{2})$ for $j=1,2,3$.
	So it suffices to show that if $\nnorm{f_j}_2=0$ for some $j=1,2,3$, then
	$$
	\lim_{N\to\infty}  \E_{n\in[N]} \, T^{a(n)}\tilde{f}_1\cdot T^{b(n)}\tilde{f}_2\cdot T^{a(n)+b(n)}\tilde{f}_3=0
	$$
	where the limit is taken in $L^2(\mu)$.
	By Theorem~\ref{T:HK} the factor $\CZ_2$ is an inverse limit of  ergodic nilsystems, so using an approximation argument
	we can assume that the system $(X, \CX, \mu,T)$ is an ergodic  $2$-step nilsystem. In this case,  property $(ii)$
	gives that the last limit is equal to
	$$
	\lim_{N\to\infty}\E_{r,s\in [N]} \, T^rf_1\cdot T^sf_2\cdot T^{r+s}f_3.
	$$
	These averages are good for degree $2$ seminorm control, which follows by a simple application of the Gowers-Cauchy-Schwarz inequality (see for example \cite[Chapter~8, Theorem~13]{HK18}),
	 hence the claim follows.
	
	Thus, we have established that degree $3$ seminorm control for the averages on the left hand side of  \eqref{E:identitysquares'}  implies degree $2$ seminorm control and identity~\eqref{E:identitysquares'} holds. Recall also that by assumption the averages on the left hand side of \eqref{E:identitysquares'}  are good for seminorm control for the system $(X,\CX,\mu,T)$.  Hence, Theorem~\ref{T:seminormdrop} applies and gives that  the averages on the left hand side of \eqref{E:identitysquares'} are good for degree $2$ seminorm control  for this system and identity~\eqref{E:identitysquares'} holds.
	
	The equivalence in the case where we restrict to  totally ergodic systems in $(i)$
	and $(ii)$ is proved in a similar fashion. We omit the details.
\end{proof}

\subsection{Proof of Theorem~\ref{T:APsRecurrence}}
For the proof of Theorem~\ref{T:APsRecurrence} we need a few preparatory results.

\begin{lemma}\label{L:equig}
	Let $a\colon \N\to \Z$ be a sequence such that for every $\ell$-step nilmanifold $Y=H/\Delta$ with  $H$  connected and simply connected and all equidistributed polynomial sequences $(g(n)\cdot e_Y)$ in $Y$,  the sequence  $(g(a(n))\cdot e_Y)$ is equidistributed in $Y$.
	Then  the sequence $a\colon \N\to \Z$ is good for $\ell$-step irrational equidistribution.
\end{lemma}
\begin{remark}
	The converse is also true and based on the fact that polynomial sequences can be represented as linear sequences on a larger nilmanifold.
	We will not use this, though, so we do not prove it.
\end{remark}
\begin{proof}
	Let $X=G/\Gamma$ be a connected $\ell$-step nilmanifold.
	Let $G_0$ be the connected component of $G$. Then, as explained in Section~\ref{SS:nilbasic}, we have  $G=G_0\Gamma$  and  $X$ can be identified with the nilmanifold $X_0:=G_0/\Gamma_0$ where $\Gamma_0:=G_0\cap \Gamma$
	(see \cite[Chapter~11, Section~4.2]{HK18}). Let $\pi\colon X\mapsto X_0$ denote
	an isomorphism.
	
	Let     $b\in G$ be arbitrary. Since $G=G_0\Gamma$,  as explained in Section~\ref{SS:nilbasic} the sequence $b^n$ can be factorized as $b^n=g_0(n)\gamma(n)$ where $g_0(n)$ and $\gamma(n)$ are polynomial sequences in $G_0$ and $\Gamma$ respectively. It follows
	that
	for every $F\in C(X)$, there exists $F_0\in C(X_0)$  (in fact $F_0:=F\circ\pi$) such that
	\begin{equation}\label{E:XX0}
		\int F\, dm_X=\int F_0\, dm_{X_0} \quad \text{ and } \quad
		F(b^n\cdot e_X)= F_0(g_0(n)\cdot e_{X_0}), \quad n\in \N,
	\end{equation}
	and conversely,   for every $F_0\in C(X_0)$ there exists $F\in C(X)$  (in fact $F:=F_0\circ\pi^{-1}$)   such that \eqref{E:XX0} holds.
	Suppose now that   $b\in G$ is ergodic in $X$ in which case by unique ergodicity we have
	$$
	\lim_{N\to\infty}\E_{n\in[N]} \,  F(b^n\cdot e_X)=\int F\, dm_X
	$$
	for every $F\in C(X)$.
	Then \eqref{E:XX0} gives that
	$$
	\lim_{N\to\infty}\E_{n\in[N]} \,  F_0(g_0(n)\cdot e_{X_0})=\int F_0\, dm_{X_0}
	$$
	for every $F_0\in C(X_0)$. Hence,
	the sequence $(g_0(n)\cdot e_{X_0})$ is  equidistributed in $X_0$. Since $X_0=G_0/\Gamma_0$ where  $G_0$ is connected and simply connected, our assumption applies (for $Y:=X_0$, $H:=G_0$, $\Delta:=\Gamma_0$) and gives  that the sequence $(g_0(a(n))\cdot e_{X_0})$ is  equidistributed in $X_0$. We deduce from this and \eqref{E:XX0} that  the sequence $(b^{a(n)}\cdot e_X)$ is equidistributed in $X$. This proves that the sequence $a\colon \N\to \Z$ is good for $\ell$-equidistribution and completes the proof.
\end{proof}
The only use of this result for our purposes is to prove the next lemma. The advantage is that in the case of a connected and simply connected nilpotent Lie group   $G$
the map $g\mapsto g^k$ is surjective for every $k\in \N$.

\begin{lemma}
	Suppose that the sequence $(da(n))$ is good for $\ell$-step irrational equidistribution for some $d\in \N$. Then the sequence $(a(n))$ is also good for $\ell$-step  irrational equidistribution
\end{lemma}
\begin{proof}
	Let $X=G/\Gamma$ be an $\ell$-step nilmanifold where  $G$ is connected and simply connected and $g(n)$ be a polynomial sequence in $G$ such that 	$(g(n)\cdot e_X)$ is   equidistributed in $X$.
	By Lemma~\ref{L:equig}, it suffices to show that the sequence  $(g(a(n))\cdot e_X)$ is   equidistributed in $X$.
	Note that since $G$ is connected, it is divisible  (see \cite[Chapter~10, Corollary~9]{HK18}), meaning that  every element has a $k$-th root in $G$ for every $k\in\N$. Hence, there exists a sequence $\tilde{g}$ in $G$ such that $\tilde{g}(dn)=g(n)$ for all $n\in\N$. It is easy to see that $\tilde{g}$ is a polynomial sequence: since $g$ can be written down as a product of sequences of the form $a^{n^i}$, the sequence $\tilde{g}$ will be a product of sequences of the form $\tilde{a}^{n^i}$ for $\tilde{a}:=a^{1/d^i}$ a $d^i$-th root of $a$.
	
	We claim that  the sequence
	$(\tilde{g}(n)\cdot e_X)$ is also equidistributed in $X$. Indeed, by Leibman's equidistribution  criterion (see Section~\ref{SS:nilbasic})  it suffices to verify this property when $G$ is a finite dimensional torus. By Weyl's equidistribution criterion,  this amounts to showing that if a $p\in \R[n]$ is a polynomial and the sequence  $(p(dn))$ is equidistributed in $\T$, then so is the sequence $(p(n))$.  This is clearly the case since  the first polynomial has a non-constant irrational coefficient if and only if the second one  does.  This proves the claim.
	
	Since
	the polynomial sequence
	$(\tilde{g}(n)\cdot e_X)$ is  equidistributed in $X$ and by assumption  the sequence $(da(n))$ is good for irrational equidistribution, Lemma~\ref{L:equig} gives that   the sequence  $(\tilde{g}(da(n))\cdot e_X)=(g(a(n))\cdot e_X)$   is equidistributed in $X$. This completes the proof.
\end{proof}
The next lemma was proved in \cite[Lemma~7.6]{FrLW06}.
\begin{lemma}
	Let $d\in \N$. Suppose that the sequence $(a(n))_{n\in\N}$ is good for $\ell$-step irrational equidistribution  and  the set $I_d:=\{n\in \N\colon d|a(n)\}$ has positive  density. Then the sequence  $(a(n))_{n\in I_d}$ is also good for $\ell$-step  irrational equidistribution.
\end{lemma}

Combining the previous two lemmas, we get the following.
\begin{lemma}\label{L:lequi}
	Let $d\in \N$. Suppose that the sequence $(a(n))_{n\in\N}$ is good for $\ell$-step irrational equidistribution  and  the set $I_d:=\{n\in \N\colon d|a(n)\}$ has positive density. Then the sequence $a_d\colon\N\to\Z$, defined by $(a(n)/d)_{n\in I_d}$,  is also good for $\ell$-step  irrational equidistribution.
\end{lemma}

We also need a similar fact about seminorm control.

\begin{lemma}\label{L:lsemi}
	Let $d, \ell, s\in \N$, $k_1, \ldots, k_\ell\in\Z$ be non-zero and distinct, and $a\colon\N\to\Z$ be a sequence. Suppose that the sequences $k_1 a, \ldots, k_\ell a$ are good for degree $s$ seminorm control for all ergodic systems,   and  the set $I_d:=\{n\in \N\colon d|a(n)\}$ has positive density. Let $a_d\colon\N\to\Z$ be the sequence defined by  $(a(n)/d)_{n\in I_d}$. Then the sequences $k_1 a_d, \ldots, k_\ell a_d$ are  good for degree  $s+1$ seminorm control for all ergodic systems.
\end{lemma}

\begin{proof}

	The argument is similar to the one used in  the proof of \cite[Lemma~5.2]{Fr21}.	
	
	An ergodic decomposition argument shows that the sequences $k_1 a, \ldots, k_\ell a$ are good
	for degree $s$ seminorm control for all systems  (not necessarily ergodic).

	We consider the system $(X_d, \CX_d, \mu_d,T_d)$ where
	$$
	X_d:=X\times \{0,\ldots, d-1\}, \quad \CX_d: = \CX\otimes \CP(\{0, \ldots, d-1\}), \quad \mu_d:=\mu\times \nu_d
	$$
	with $\nu_d:=\frac{\delta_0+\cdots+\delta_{d-1}}{d}$, and for $x\in X$, we let
	$$
	T_d(x,i):=\begin{cases}(x,i+1), \, &i=0,\ldots, d-2,\\
		(Tx,0),\; &i=d-1.
	\end{cases}
	$$
 The system $(X_d, \CX_d, \mu_d,T_d)$ is often referred to as the \textit{$d$'th root of $(X, \CX, \mu, T)$}.
	Its key property is that
	$T_d^d(x,i)=(Tx,i)$ for all $x\in X$ and $i\in \{0,\ldots, d-1\}$, hence for $f\in L^\infty(\mu)$, we have
	\begin{equation}\label{E:Td}
		(T_d^d(f\otimes 1))(x,i)=(Tf)(x).
	\end{equation}
	
	Applying our good seminorm assumption  for the product system $$(X_d\times X_d, \CX_d\otimes\CX_d, \mu_d\times\mu_d,T_d\times T_d),$$
	we get   that 	if  $g_1, \ldots, g_\ell\in L^\infty(\mu_d)$ and $\nnorm{g_j\otimes\overline{g_j}}_{s,T_d\times T_d}=0$ for some $j\in [\ell]$, then
	\begin{align}
		\lim_{N\to\infty}\E_{n\in [N]}\,  (T_d\times T_d)^{k_1a(n)}(g_1\otimes \overline{g_1})\cdots (T_d\times T_d)^{k_\ell a(n)}(g_\ell\otimes\overline{g_\ell})= 0.
	\end{align}
Using the bound \eqref{E:T vs TxT}, Lemma~\ref{L:product trick}(ii) for $w_n:={\bf 1}_{I_d}(n)$, $n\in \N$,  and the fact that $I_d$ has positive density, we deduce that
	\begin{equation}\label{E:assumption}
		\lim_{N\to\infty}\E_{n\in I_{d, N}}\,  T_d^{k_1a(n)}g_1\cdots T_d^{k_\ell a(n)}g_\ell= 0
	\end{equation}
	in $L^2(\mu_d)$ whenever $\nnorm{g_j}_{s+1,T_d}=0$ for some $j\in [\ell]$, where for $N\in \N$ we let $I_{d,N}:=I_d\cap [N]$.
	
	We claim that for  this $s\in \N$,  the sequence $(a(n)/d)_{n\in I_d}$
	is good for  degree  $s+1$ seminorm control along $\ell$-term APs   for the system $(X, \CX, \mu,T)$. To see this, let $f_1,\ldots, f_\ell\in L^\infty(\mu)$ be such that  $\nnorm{f_j}_{s+1,T}=0$ for some $j\in [\ell]$, say for $j=j_0$. It suffices to show that
	$$
	\lim_{N\to\infty}\E_{n\in I_{d, N}}\,  T^{k_1a(n)/d}f_1\cdots T^{k_\ell a(n)/d}f_\ell= 0
	$$
	in $L^2(\mu)$.
	Let $g_j\in L^\infty(\mu_d)$ be defined by
	$g_j:=f_j\otimes 1$ for each $j=1,\ldots, \ell$, and recall that $\mu_d=\mu\times \nu_d$. By \eqref{E:Td},  it suffices to show that
	$$
	\lim_{N\to\infty}\E_{n\in I_{d, N}}\,  T_d^{dk_1a(n)/d}g_1\cdots T_d^{dk_\ell a(n)/d}g_\ell= 0
	$$
	in $L^2(\mu_d)$, or equivalently that
	$$
	\lim_{N\to\infty}\E_{n\in I_{d, N}}\,  T_d^{k_1a(n)}g_1\cdots T_d^{k_\ell a(n)}g_\ell= 0.
	$$
	By \eqref{E:assumption}, this holds as long as we establish that $\nnorm{g_{j_0}}_{s+1,T_d}=0$.
	To see this, note that
	since $g_{j_0}=f_{j_0}\otimes 1$ and $\mu_d=\mu\times \nu_d$, our assumption $\nnorm{f_{j_0}}_{s+1,T}=0$  and  \eqref{E:Td} give that
	$\nnorm{g_{j_0}}_{s+1,T_d^d}=\nnorm{f_{j_0}}_{s+1,T}=0$. The implication
	$\nnorm{g_{j_0}}_{s+1,T_d}=0$ then follows from \eqref{E: seminorm of powers}. This completes the proof.
\end{proof}
We are now ready to prove Theorem~\ref{T:APsRecurrence},  we restate it for convenience.
\begin{theorem}\label{T:APsRecurrence'}
	If the sequence  $a\colon \N\to \Z$  is  good for seminorm control along $\ell$-term APs, is good for  $\ell$-step irrational equidistribution, and
	has good divisibility properties, then it
	is good for recurrence along $\ell$-term APs.
\end{theorem}
\begin{proof}	
	Let $(X, \CX, \mu,T)$ be a system and $f\in L^\infty(\mu)$ be such that $f\geq 0$ and  $\int f\, d\mu>0$. Let  also $k_1,\ldots, k_\ell\in \Z$. Our goal is to show that for some subset  $I\subset \N$ with positive density, we have
	\begin{equation}\label{E:positive}
		\lim_{N\to\infty}\E_{n\in I_{N}} \int f\cdot T^{k_1a(n)}f\cdots T^{k_\ell a(n)} f\, d\mu>0,
	\end{equation}
	 where as before, for $N\in \N$ we let  $I_N:=I\cap [N]$.
	We can assume that $k_1,\ldots, k_\ell$ are non-zero and distinct and the system  is ergodic.
	
	Since the sequence $(a(n))$ is good for seminorm control along $\ell$-term APs, and the set $I$ is assumed to have positive density, Lemma~\ref{L:product trick}(ii) for $w_n:={\bf 1}_I(n)$, $n\in \N$, and the bound  \eqref{E:T vs TxT} imply that it is also good for seminorm control of the   averages in \eqref{E:positive}. Hence,  using Theorem~\ref{T:HK}, we can assume that the system is an inverse limit of ergodic nilsystems,
	and using an approximation argument (as in \cite[Lemma~2.1]{BLL08} or \cite[Lemma~3.2]{FuK79}), we can assume that it is an ergodic nilsystem. As explained in Section~\ref{SS:nilbasic}, there exists $d\in \N$ such that $X$ can be partitioned as  $X=\bigcup_{i=0}^{d-1} T^iX_0$, and for $i=0,\ldots, d-1$, the transformation $T^d$ leaves the set $T^iX_0$ invariant and the corresponding system is totally ergodic.
	
	Since $\int f\, dm_{X}>0$ and $X=\bigcup_{i=0}^{d-1} T^iX_0$, there exists 	$i_0\in \{0,\ldots, d-1\}$ such that   $\int f\, dm_{T^{i_0}X_0}>0$. We  let
	$$
	I_d:=\{n\colon d| a(n)\}.
	$$
	We claim that
	for  $I:=I_d$ the positivity  property \eqref{E:positive} holds, or equivalently, that
	\begin{equation}\label{E:positiveX}
		\lim_{N\to\infty}\E_{n\in I_{d, N}} \int f\cdot T^{k_1a(n)}
		f\cdots T^{k_\ell a(n)} f\, dm_X>0
	\end{equation}
where for $N\in \N$ we let $I_{d,N}:=I_d\cap [N]$.
	
	Let $S:=T^d$. For $i\in [d]$, the transformation  $S$ leaves the set $T^iX_0$ invariant, and the system $(T^iX_0,m_{T^{i}X_0},S)$ is totally ergodic.
	Let $(a_d(n))$ be the sequence defined by
	$(a(n)/d)_{n\in I_d}$. Since the set $I_d$ has positive density,  it suffices to show that
	\begin{equation}\label{E:positive'}
		\lim_{N\to\infty}\E_{n\in[N]} \int f\cdot S^{k_1a_d(n)}
		f\cdots S^{k_\ell a_d(n)} f\, dm_{T^{i_0}X_0}>0.
	\end{equation}
	
	Our assumptions and Lemma~\ref{L:lequi} and \ref{L:lsemi} imply that the sequence $(a_d(n))$ is good for $\ell$-step irrational equidistribution and
	good for  seminorm control along $\ell$-term APs  for all ergodic systems.
	So we are now in position to use the implication $(iii)\implies (i)$ in  Theorem~\ref{T:APsConvergenceTE'}.  It gives  that the sequence 	$(a_d(n))$ is good for $\ell$-convergence for all totally ergodic systems. We apply this for the totally ergodic nilsystem $(T^{i_0}X_0,m_{T^{i_0}X_0},S)$, and we get that the average in \eqref{E:positive'}  equals
	$$
	\lim_{N\to\infty}\E_{n\in[N]} \int f\cdot S^{k_1n}
	f\cdots S^{k_\ell n} f\, dm_{T^{i_0}X_0}.
	$$
	Since  $f\geq 0$ and $\int f\, dm_{T^{i_0}X_0}>0$, the last limit  	 is positive by Furstenberg's multiple recurrence result \cite{Fu77}.
	This completes the proof.
\end{proof}

\end{document}